\documentclass[11pt]{amsart}
\usepackage{amsmath,amsfonts,amssymb,amsxtra,amscd,enumerate,amsthm}
\usepackage{color}
\usepackage{latexsym}
\usepackage{epsfig}
  \usepackage{fullpage}

\newcommand\les{\lesssim}
\newcommand\ges{\gtrsim}

\newcommand\R{\mathbb{R}}

\newcommand\C{\mathbb{C}}
\newcommand\Z{\mathbb{Z}}
\newcommand\N{\mathbb{N}}
\renewcommand\S{\mathbb{S}}

\newcommand{\calE}{\mathcal E}
\newcommand{\calT}{\mathcal T}
\newcommand{\calH}{\mathcal H}
\newcommand{\calR}{\mathcal R}
\newcommand{\calN}{\mathcal N}

\newcommand{\ls}{{\lesssim}}

\newtheorem{theo}{Theorem}
\numberwithin{theo}{section} 
\newtheorem{lema}[theo]{Lemma}
\newtheorem{prop}[theo]{Proposition}

\newtheorem{conjecture}[theo]{Conjecture}

\newtheorem{con}{Condition}

\numberwithin{equation}{section}

\begin{document}
\title{The almost optimal multilinear restriction estimate for hypersurfaces with curvature: the case of $n-1$ hypersurfaces in $\R^n$}

\author[I. Bejenaru]{Ioan Bejenaru} \address{Department
  of Mathematics, University of California, San Diego, La Jolla, CA
  92093-0112 USA} \email{ibejenaru@math.ucsd.edu}

\begin{abstract} In this paper we establish the optimal multilinear restriction estimate for $n-1$ hypersurfaces with some curvature, where $n$ is the dimension of the underlying space. The result is sharp up to the endpoint and the role of curvature is made precise in terms of the shape operator.   
\end{abstract}


\maketitle

\section{Introduction}

For $n \geq 2$, let $U \subset \R^{n-1}$ be an open, bounded and connected  neighborhood of the origin and let $\Sigma: U \rightarrow \R^{n}$ be a smooth parametrization of an $n-1$-dimensional submanifold of $\R^{n}$ (hypersurface), which we denote by $S=\Sigma(U)$. By a smooth parametrization we mean that $\Sigma$ satisfies
\begin{equation} \label{smooth}
\| \partial^\alpha \Sigma \|_{L^\infty(U)} \lesssim_\alpha 1,
\end{equation} 
for $|\alpha| \leq N$ for some large $N$. We then say that $S$ is a smooth hypersurface if it admits a parametrization satisfying \eqref{smooth}.
To such a parametrization of $S$ we associate the extension operator $\calE$  defined by
\[
\calE f(x) = \int_U e^{i x \cdot \Sigma(\xi)} f(\xi) d\xi. 
\]
where $f \in L^1(U)$ and $x \in \R^n$. This operator is closely related to a more intrinsic formulation of the extension operator:
\[
\tilde \calE g(x)= \int_S e^{ix \cdot \omega} g(\omega) d \sigma_S(\omega),
\]
where $g \in L^1(S; d \sigma_S)$ and $x \in \R^n$. Indeed, using the parametrization $\Sigma$ as above we obtain
$\omega=\Sigma(\xi)$ and $d \sigma_S(\omega)=|\frac{\partial \Sigma}{\partial_{\xi_1}} \wedge ... \wedge \frac{\partial \Sigma}{\partial_{\xi_{n-1}}}| d \xi$, thus
with $f(\xi)= g(\Sigma(\xi)) |\frac{\partial \Sigma}{\partial_{\xi_1}} \wedge ... \wedge \frac{\partial \Sigma}{\partial_{\xi_{n-1}}}|$, the two formulations are equivalent. 
We will use good parameterizations in the sense that $ |\frac{\partial \Sigma}{\partial_{\xi_1}} \wedge ... \wedge \frac{\partial \Sigma}{\partial_{\xi_{n-1}}}| \approx 1$
throughout the domain and, given that all results are in terms of $L^p$ norms, the equivalence is carried out at the levels of results as well. 

Given $k$ smooth, compact hypersurfaces $S_i \subset \R^{n}, i=1,..,k$, where $1 \leq k \leq n$, the $k$-linear restriction estimate is the following
inequality 
\begin{equation}  \label{MRE}
\| \Pi_{i=1}^k \calE_i f_i \|_{L^p(\R^{n})} \les \Pi_{i=1}^k \| f_i \|_{L^2(U_i)}.  
\end{equation}
In a more compact format this estimate is abbreviated as follows:
\begin{equation} \label{Rsp}
\calR^*(2 \times ... \times 2 \rightarrow p).
\end{equation}
A natural condition to impose on the hypersurfaces is the standard transversality condition: there exists $\nu >0$ such that
\begin{equation} \label{trans}
| N_1(\zeta_1) \wedge ... \wedge N_{k}(\zeta_k) | \geq \nu,
\end{equation}
for any $\zeta_i \in S_i, i \in \{1,..,k\}$; here $N_i(\zeta_i)$ is the unit normal at $\zeta_i$ to $S_i$ and the choice of orientation is not important. 

There are two main conjectures regarding the optimal exponent in \eqref{Rsp}. For the generic case, when only the transversality condition
is assumed, the optimal conjectured exponent for \eqref{Rsp} is $p=\frac2{k-1}$. With the exception of the end-point case this problem is well-understood by now. Bennett, Carbery and Tao  \cite{BeCaTa} established the near-optimal version of \eqref{Rsp}, Guth \cite{Gu-main} has proved the end-point mutlilinear Kakeya analogue of \eqref{Rsp} and, very recently, Tao \cite{Tao-new} has established \eqref{Rsp} up to the end-point, that is for any $p>\frac2{k-1}$. The optimal problem $p=\frac2{k-1}$ is currently open. 

There is also the non-generic case where, in addition to the transversality condition, one also assumes appropriate curvature conditions on the hypersurfaces $S_i$ and in this case the following result is conjectured:
\begin{conjecture} \label{conjj}
Under appropriate transversality and curvature conditions on the surfaces $S_i$, $\calR^*(2 \times ... \times 2 \rightarrow p)$ 
holds true for any $p \geq p(k)=\frac{2(n+k)}{k(n+k-2)}$. 
\end{conjecture}

We note that the case $k=1$ is the classical Tomas-Stein result, see for instance \cite{St}. The case $k=n$ brings no improvement on the exponent $p$ over the generic case (same optimal exponent of $\frac2{n-1}$) and this case is essentially covered by the works mentioned earlier. 

The conjecture for $k=2$  was formulated in \cite{FoKl} by Foschi and Klainerman in the context of paraboloid and cone, that is when $S_1$ and $S_2$ are both subsets of the paraboloid or the cone. For $k \geq 3$, the conjecture was formulated by Bennett in \cite{Ben} in the context of hypersurfaces with everywhere positive curvature. In \cite{Be3} we formalized the general statement of the above conjecture with a complete description of the role of geometry in this problem.

As for its resolution, the bilinear case ($k=2$) of the conjecture is very well understood, while the case for higher levels of multilinearity $3 \leq k \leq n-1$ has been open.

The bilinear case was intensely studied, see \cite{Bou-CM, Wo, Tao-BW, Tao-BP, TV-CM1, Lee-BR,LeeVa, Be2, Va, Le, Te, Sto, BuMuVa, BaLeLe, TaVaVe} and references therein. We should highlight the works of Wolff \cite{Wo} and Tao \cite{Tao-BP} where the conjectured was established, up to the endpoint, for subsets of the cone, respectively paraboloids.  Currently the problem is solved in the regime $p > \frac{n+3}{n+1}$ for general hypersurfaces with curvature; the end-point $p = \frac{n+3}{n+1}$ is solved only for the subsets of the cones, see Tao \cite{Tao-BW}.

There have been many applications of the bilinear restriction theory to other problems in Harmonic Analysis:  linear restriction conjecture by Tao, Vargas and Vega \cite{TaVaVe} and Tao \cite{Tao-BP}, Falconer distance problem by Erdogan \cite{Er}, bounds for Bochner-Riesz operators by Lee \cite{Lee-B-R}, Schr\"odinger maximal function by Lee \cite{Lee-SMF}; here we just named a few applications that we are aware of. 

In PDE the bilinear restriction theory has been used for very different problems: profile decomposition adapted to the linear Schr\"odinger equation by Merle and Vega \cite{MeVe}, threshold conjecture for Wave Maps by Sterbenz and Tataru \cite{StTa} and Krieger and Schlag \cite{KrSc}, estimates for eigenfunctions of the Laplace-Beltrami operator on manifolds by Blair and Sogge \cite{blSo},  global well-posedness for the Dirac-Klein-Gordon system with resonant masses by Candy and Herr \cite{CaHe}, uniqueness in the Calder\'on problem by Ham, Kwon and Lee \cite{HaKwLe}; this list is by no means exhaustive. 

The $k$-linear estimate with $3 \leq k \leq n-1$ has been studied by the author in \cite{Be3,Be4} where the Conjecture \ref{conjj} had been established up the end-point for a particular class of hypersurfaces; while this class has nice properties for the purpose of proving the conjecture, it does not cover the most interesting examples of hypersurfaces such as the paraboloid, sphere, cone etc. Guth has considered a weaker formulation of the conjecture for the  paraboloid in \cite{Gu-II}; although weaker than the statement of the conjecture, Guth's result is strong enough to derive the same conclusion for the linear restriction conjecture as one would do if the result of the conjecture were available. 

In this paper we look at the particular case when $k=n-1$ but in a very general geometric setup. We begin by formalizing the conditions we impose on our surfaces. As before, $S_i, i \in \{1,..,k \}$ are hypersurfaces with smooth parameterizations $\Sigma_i: U_i \subset \R^{n-1} \rightarrow \R^{n}$, where each $U_i$ is open, bounded and connected  neighborhood of the origin.  

The first condition we impose on our hypersurfaces is the standard transversality condition \eqref{trans}.

The second condition we impose is related to curvature: there exists  $\nu_1 >0$ such that for any $\zeta_i \in S_i, i \in \{1,..,k\}$, 
and for any $l \in \{1,..,k\}$, the following holds true: 
\begin{equation} \label{curva}
| N_1(\zeta_1) \wedge ... \wedge N_{k}(\zeta_k) \wedge S_{N(\zeta_l)} v | \geq \nu_1 | v|,
\end{equation}
for every $v \in T_{\zeta_l} S_l$ with the property $v \perp N_i(\zeta_i), \forall i=1,..,k$. Based on simple continuity arguments, it can be easily shown that there is some flexibility in the above estimate in the sense that it holds true for any $v \in T_\zeta S_l$ such that $|v \wedge N_1(\zeta_1) \wedge ... \wedge N_{k}(\zeta_k)| \geq (1-\nu_2)  |v| |N_1(\zeta_1) \wedge ... \wedge N_{k}(\zeta_k)|$ for some $0 < \nu_2 \ll 1$, but this will not play a role in our analysis. 

The main result of this paper is the following
\begin{theo} \label{mainT}  Let $k=n-1$. Assume that $S_1,..,S_k$ satisfy \eqref{trans} and \eqref{curva}. 
Given any $p$ with $p(k)=\frac{2(n+k)}{k(n+k-2)}< p \leq \infty$, the following holds true
\begin{equation} \label{mainE}
\| \Pi_{i=1}^k \calE_i f_i \|_{L^p(\R^{n+1})} \leq C(p) \Pi_{i=1}^k \| f_i \|_{L^2(U_i)}, \quad \forall f_i \in L^2(U_i).   
\end{equation}
\end{theo}
It is an easy exercise to check that this theorem covers the interesting examples when $S_i$'s are all subsets of the paraboloid, sphere or cone 
subjected to the transversality condition \eqref{trans}. Thus, in the case $k=n-1$, the above result settles down the answer to Conjecture \ref{conjj}, up to the end-point, in a fairly complete fashion. First, it contains the interesting models just mentioned and it goes beyond these models to a general setup. Second, it highlights the clear role that curvature plays into this problem. Here are the few take-aways from the curvature condition \eqref{curva}:

- the geometry is more complicated than simply counting the non-zero principal curvatures: one can obtain similar results for hypersurfaces with all principal curvatures being non-zero as well as for hypersurfaces with $k-1$ principal curvatures being zero; this was known from the bilinear restriction estimate where similar results were derived for disjoint subsets of the paraboloid and cone.

- the picture becomes more subtle as it is not enough to simply obtain the result by imposing that the principal curvatures corresponding to principal directions transversal to $N_1,..,N_k$ are non-zero (loosely speaking); in \cite{Lee-BR} Lee had shown that the optimality claimed by the Conjecture \ref{conjj} may not hold true in the bilinear case for disjoint subsets of hypersurfaces with principal curvatures of different signs (but with non-zero Gaussian curvature).

- \eqref{curva} essentially says two things: the hypersurfaces need to have non-degenerate curvature in the directions (from $T_\zeta S_i$) that are transversal to $N_1,..,N_k$ (any sample of normals from $S_1,..,S_k$) and that the shape operator keeps that transversality when acting on those directions. 

Taking into account this result and the author's prior work in \cite{Be3,Be4}, we are lead to believe that \eqref{curva} (or equivalent formulations of) 
together with the transversality condition \eqref{trans} are the appropriate geometric conditions for the Conjecture \eqref{conjj} to hold true in the remaining open cases $3 \leq k \leq n-2$. 

The current state of the Conjecture \ref{conjj} is as follows. If $n=3$, then, up to the end-point, things had been settled down before this work in all cases 
$k \in \{1,2,3\}$, as described earlier in this introduction. If $n=4$, then, up to the end-point, the results in the cases $k \in \{1,2,4\}$ were known from before and our paper establishes the case $k=3$, thus completing the picture of all levels of multilinearity. In the cases $n \geq 5$, gaps in the theory 
are still in place, namely for $3 \leq k \leq n-2$.  

An interesting consequence of the above result is that it provides also the sharp $L^{\frac2{k-1}}$ (generic) estimate in the case $k=n-1$ for a fairly large class of hypersurfaces and this is a new result in the literature. Previously this type of sharp result was known only for the bilinear estimate in $L^2$, in the very general case setup for the hypersurfaces considered. However we note that the more difficult and open problem is the case $k=n$ and this is not addressed by our results.

In a nutshell, the main ideas of proof are polynomial partitioning, the generic multilinear estimate in a localized setting and an analysis based on differential geometry that reveals the role of the shape operator in the problem. 

In the context of Kakeya conjecture, the polynomial partitioning was  introduced by Dvir in \cite{Dv} who solved the Kakeya conjecture on finite fields. 
Guth introduce the polynomial partitioning in the continuum setting in \cite{Gu-main} where he solved the multilinear Kakeya conjecture. In \cite{Gu-I} Guth initiated  the study of the restriction problem in $\R^3$ using the polynomial partitioning. In \cite{Gu-II} Guth adapted those ideas to the restriction problem in $\R^n$ and makes significant progress in the range of allowed exponents. Later works on the subject of the linear restriction conjecture have successfully used the polynomial partitioning to improve the range of exponents, see \cite{De,HiRo,Wa}.

In this work we employ the polynomial partitioning along the lines developed by Guth in \cite{Gu-II}. We essentially use it as a tool that either allows the use of a direct induction argument or helps to extract localization properties.

The next key idea revolves around the generic multilinear restriction estimate, when no curvature assumptions are made on the hypersurfaces involved.  In \cite{Be1}, we initiated the study of the effect of localization in the multilinear restriction estimate - this result played a crucial role in our prior works \cite{Be3} and \cite{Be4} on the multilinear restriction estimate for conical type hypersurfaces with curvature. 
One drawback of the theory developed in \cite{Be1} was that the localization were required to be "flat" and this was acceptable in the context of our works in \cite{Be3,Be4}. In the current context we need to accommodate localizations in neighborhoods of manifolds which now are allowed to "curve"
and the techniques introduced in \cite{Be1} do not suffice. A more robust result has been recently established by the author in \cite{Be5} and this suffices for our purposes, see  Theorem \ref{MB} for details. 

In our prior works \cite{Be2,Be3,Be4} we began a systematic analysis of the role of the geometry in the multilinear restriction theory. The upshot of these works was that the shape operator is the right object to describe the role of curvature in this problem. In this paper we further refine the analysis of the shape operator from our previous works so as to fit the most general framework. A novelty is that this analysis has to be compatible with elements coming from the polynomial partitioning arguments. 

A natural question to ask is why are our arguments limited to the case $k=n-1$. The simple answer is that by using once the polynomial partitioning argument directly on the estimate which we aim to prove, we reduce the problem to the case of a multilinear restriction estimate in the generic case with some additional localization properties - things are a little more complicated, but this states the main idea. In \cite{Gu-II}, Guth runs the induction on scale argument and the polynomial partitioning on a modified statement, that essentially acts as a way to reiterate the use of polynomial partitioning. 
In closing the argument, Guth uses in an essential way the fact that the estimate is a linear one, and not a multilinear one, which is our case. 
We could not find a similar proper substitute to our main estimate that will play a similar role and allow us to cover the cases $3 \leq k \leq n-2$. 

The paper is organized as as follows. In this section we continue with some notation and then recall two results from a recent paper of the author \cite{Be5} which will play an important role in supporting our analysis here; note however that both results from \cite{Be5} are about the generic multilinear restriction estimate. 
In Section \ref{tdg} we introduce some basic notation and results about the shape operator and then cover some geometric results that will be instrumental in 
highlighting the role of the geometry in this problem. In Section \ref{swp} we introduce the wave packet theory. In Section \ref{TAG} we go over the main tools from algebraic geometry that we use in this paper - this section can be seen as a summary of what we need from \cite{Gu-II}. In Section \ref{sma} we cover the proof of 
Proposition \ref{indp} from which the main result of this paper, Theorem \ref{mainT}, follows. 

\subsection{Notation}

We make the convention that the surfaces involved have very small diameter in the sense that if $S=\Sigma(U)$ is a parametrization then $U \subset \R^{n-1}$ has small diameter in the classical sense. In particular, for every $i=1,..,k$ the following holds true:
\begin{equation} \label{cN}
|N_i(\zeta_1) - N_i(\zeta_2)| \leq c, \quad \forall \zeta_1, \zeta_2 \in S_i,
\end{equation}
where $0 < c \ll 1$ is a constant small enough. This assumption does not affect our analysis: we can break each surface in finitely many pieces 
of small diameter, run the argument with the above setup and then sum up the result on these pieces. 

If $v_1,..,v_m$ are vectors in $\R^n$, by $span(v_1,..,v_m)$ we mean the standard subspace of $\R^n$ spanned by $v_1,..,v_m$.

The Fourier transform of a Schwartz function $f:\R^n \rightarrow \C$ is defined by
\[
\mathcal{F}(f)(\xi)=\hat f(\xi)= \int e^{-i x \cdot \xi} f(x) dx. 
\]
The inverse Fourier transform is defined by 
\[
\mathcal{F}^{-1}(g)(x)=\check g(x)= \frac1{(2\pi)^n} \int e^{i x \cdot \xi} g(\xi) d\xi.
\]
These definitions are then extended to distributions, in particular to $L^p(\R^n)$ spaces, in the usual manner.

We use the standard notation $A \ls B$, meaning $A \leq C B$ for some universal $C$ which is independent of variables used in this paper. By $A \ls_N B$ we mean $A \leq C(N) B$ and indicate that $C$ depends on $N$.  

\subsection{Two results concerning the generic multilinear restriction estimate} In this section we recall two results from a recent work of the author \cite{Be5}. Our proofs in this paper will record an $\epsilon$-loss in the sense that we obtain \eqref{mainE} with a factor of $R^\epsilon$ when estimating the right-hand side on balls of radius $R$. The first result we quote from \cite{Be5} provides the necessary $\epsilon$-removal ingredient that allows us to obtain \eqref{mainE} such a loss. We note that the condition \eqref{curva} implies that 
\begin{equation} \label{dec}
|\calE_i \psi| \lesssim C(\psi) (1+|x|)^{-\frac12}
\end{equation}
for any smooth $\psi$ supported within $U_i$. This is a simple consequence of following: \eqref{curva} implies that at least one the principal curvatures of $S_i$ is non-zero, therefore $S_i$ is of $2$-type, in the language used in  \cite{St} (see chapter 8, section 3.2); \eqref{dec} follows from \cite{St} (see chapter 8, section 3.2, Theorem 2). Now we can invoke Theorem $1.1$ in \cite{Be5} in our particular setup to obtain the following

\begin{prop} \label{epsr} We assume that the hypersurfaces $S_i, i=1,..,k$ satisfy the transversality condition \eqref{trans} and \eqref{curva}. 
 If $\mathcal{R}(2 \times ... \times 2 \rightarrow p,\epsilon)$
holds true for some $p \geq \frac2k$, then $\mathcal{R}(2 \times ... \times 2 \rightarrow q)$ holds true for any $q \geq p + (n+p+1)\frac{C}{\log \frac1\epsilon}$, where $C$ is any constant satisfying $C > \min(2,n-1)$. 
\end{prop}

The second result from \cite{Be5} is about a refinement of the generic multilinear restriction estimate under some localization hypotheses. 
Given a $S$ a submanifold of $\R^n$, we let $B_\epsilon(S)=\cup_{\zeta \in S} B_\epsilon(\zeta)$ and denote by $N_\zeta S$ the normal plane to $S$ at $\zeta$.
One observations is in place here: we have two entities, $N_i(\zeta_i)$ and $N_{\zeta_i} S_i$, carrying similar but different notation - the first is the normal (unit vector) to $S_i$ at $\zeta_i$, while the other one is the normal plane to $S_i$ at $\zeta_i$. 

 If $V_1,..,V_k$ are $d_k$-dimensional planes, then by $|V_1 \wedge ... \wedge V_k|$ we mean the quantity
$|v_{1,1} \wedge ... v_{1,d_1} \wedge ... \wedge v_{k,1} \wedge ... \wedge v_{k,d_k}|$ where $v_{i,1},..,v_{i,d_i}$ is an orthonormal basis in $V_i$; it is easily seen that the defined quantity is independent of the choices of orthonormal systems.

The small support condition is the following:

\begin{con} \label{C} Assume that we are given submanifolds $S_i' \subset S_i, i =1,..,k$, of codimension $c_{i}$,
with the property that  there exists $\nu >0$ such that
\begin{equation} \label{normal6}
| N_{\zeta_1} S_1' \wedge  .. \wedge N_{\zeta_k} S_k' | \geq \nu
\end{equation}
for all choices $\zeta_i \in S_i'$. Given $g_i \in L^2(S_i), \forall i =1,..,k$, we assume that 
 $supp g_i \subset  B_{\mu_i} (S_i') \cap S_i$, where $0 < \mu_i \ll 1$.
 \end{con}
 
We make the observation the total number of directions in which localization is provided cannot exceed $n-k$, that is $c_1+..+c_k \leq n-k$; this follows from \eqref{normal6}.

If the assumptions above are imposed in the generic multilinear estimate, we obtain the following result.
\begin{theo} \label{MB}
Assume $S_i, i=1,..,k$ are smooth. In addition, assume that $g_1,..,g_k$ satisfy Condition \ref{C}.
Then for any $\epsilon > 0$, there is $C(\epsilon)$ such that 
the following holds true
\begin{equation} \label{Lf}
\| \Pi_{i=1}^{k} \tilde \calE_i g_i \|_{L^\frac{2}{k-1}(B(0,R))} \leq C(\epsilon) (\Pi_{j=1}^k \mu_j)^{\frac12} R^\epsilon \Pi_{i=1}^{k} \| g_i \|_{L^2(S_i)}.
\end{equation}
\end{theo}

\subsection*{Acknowledgments}

Part of this work was supported by an NSF grant, DMS-1900603.

\section{Tools from differential geometry} \label{tdg}

We recall some basic facts about the shape operator that can be found in more detail in any classic differential geometry textbook, see for instance \cite{doCa}. 

Given a hypersurface  $S \subset \R^{n}=\{(\xi,\tau): \xi \in \R^{n-1}, \tau \in \R\}$ parametrized by $\tau=\varphi(\xi), \xi \in D \subset \R^{n-1}$, we define the Gauss map $g: S \rightarrow \S^{n-1} \subset \R^{n}$ ($\S^{n-1}$ is the unit sphere in $\R^{n}$) by 
$N(\zeta)=\frac{(-\nabla \varphi(\xi), 1)}{|(-\nabla \varphi(\xi), 1)|}$ where $\zeta=(\xi,\varphi(\xi)) \in S$. Since
$T_\zeta(S)$ and $T_{g(\zeta)} \S^{n-1}$ are parallel, we can identify them, and define $ d g_\zeta: T_\zeta S \rightarrow T_\zeta S$
by $d g_\zeta v= \frac{d}{dt} (N (\gamma(t)))|_{t=0}$ where $\gamma \subset S$ is a curve with $\gamma(0)=\zeta, \gamma'(0)=v$.  
The shape operator $S_{N(\zeta)}: T_\zeta S \rightarrow T_\zeta S$ is defined by 
\[
S_{N(\zeta)} = - d g_\zeta,
\]
 where we keep the subscript 
$N(\zeta)$ to indicate that the shape operator depends on the choice of the normal vector field at $S$. It is known that $S_{N(\zeta)}$ is symmetric, therefore there exists an orthonormal
basis of eigenvectors $\{e_i\}_{i=1,n-1}$ of $T_\zeta S$ with real eigenvalues $\{\lambda_i\}_{i=1,n-1}$. Locally $S$ is orientable and we assume
a consistency with the orientation in $\R^{n}$, that is $\{e_1,..,e_{n-1}\}$ is a basis in the orientation of $S$ and $\{e_1,..,e_{n-1},N(\zeta)\}$ 
is a basis in the orientation of $\R^{n}$. Then $e_i$ are the principal directions and $\lambda_i=k_i$ are the principal curvatures of $S$ (to be more precise they are the curvatures of the embedding $id: S \rightarrow \R^{n}$, where $id$ is the identity mapping). The Gaussian curvature is defined by $det S_N=\Pi_{i=1}^{n-1} \lambda_i$.

As it can be seen from the above notation, we use the letter $S$ to denote both hypersurfaces and the shape operator. 
However the shape operator notation will always occur together with the normal, that is $S_N$, and this clarifies any ambiguity in what $S$ stands for. 

Using the inclusion $i: S \rightarrow \R^n$, we obtain that $T_\zeta S \subset T_\zeta \R^n$ and by the standard identification $T_\zeta \R^n \cong \R^n$.  Using this identification allows us to write down various objects such as $N_1(\zeta_1) \wedge ... \wedge N_{k}(\zeta_k) \wedge S_{N(\zeta)} v$ that appears in \eqref{curva}. This identification will be tacitly used throughout the rest of the paper. 

Given any smooth hypersurface we note that for every $m \in \{1,..,n-1\}$, there exists $C_m$ such that the following holds true:
\begin{equation} \label{volc}
| S_{N(\zeta)} v_1 \wedge ... \wedge S_{N(\zeta)} v_m | \leq C_m | v_1 \wedge ... \wedge v_m |,
\end{equation}
for any $v_1,..,v_m \in T_\zeta S$. This follows from \eqref{smooth}; 
in fact $C_m$ depends only on the derivatives up to second order in \eqref{smooth}. 

If $S$ is a (more general) submanifold of $\R^n$, then we use $N_\zeta S$ to denote the normal plane at $\zeta \in S$. In the case when
$S$ is a hypersurface, $N(\zeta)$, the choice of unit normal to $S$ at $\zeta$ mentioned above, is the unit vector that spans $N_\zeta$. 
This hopefully clarifies any confusion that may occur due to the similarity in notation between the two entities.  

Before getting to the main result in this section, we record the following technical result. 

\begin{lema} \label{vols}
i) Assume that $v_1,..,v_m \in \R^n$ are unit vectors such that $|v_1 \wedge .. \wedge v_m| \leq c^m$ for some $0 \leq c \leq 1$. Then there exists a unit vector $\alpha=(\alpha_1,..,\alpha_m) \in \R^m$ 
with the property that
\begin{equation}
|\sum_{i=1}^m \alpha_i v_i| \leq c.
\end{equation}

ii) Assume that $A \in M_{n \times n}(\R)$ is a symmetric matrix with the property that $|\det A| \leq \epsilon$. Then there exists a unit vector $v \in \R^n$ with the property 
that $|Av| \leq \epsilon^\frac1n$. 

\end{lema}

\begin{proof} i) We proceed by induction with respect to $m$. The case $m=1$ is obvious. We assume the result for $m < n$ and prove it for $m+1$.
Thus we are given $m+1$ vectors with the property $|v_1 \wedge .. \wedge v_{m+1}| \leq c^{m+1}$. We write
\[
v_{m+1}=v_{m+1}^\perp + \sum_{i=1}^m \beta_i v_i. 
\]
where $\sum_{i=1}^m \beta_i v_i$ is the projection of $v_{m+1}$ onto $span(v_1,..,v_m)$ while $v_{m+1}^\perp$ is the component of $v_{m+1}$
in $span(v_1,..,v_m)^\perp \subset span(v_1,..,v_m, v_{m+1})$. Note that
\[
|v_1 \wedge .. \wedge v_{m+1}| = |v_1 \wedge .. \wedge v_m \wedge v^\perp_{m+1}|= |v_1 \wedge .. \wedge v_m| \cdot |v^\perp_{m+1}|. 
\]

If $|v_{m+1}^\perp| \geq c$, then $|v_1 \wedge .. \wedge v_m| \leq c^m$ we can apply the induction hypothesis to obtain the desired result; note that 
the induction gives $\alpha \in \R^m$ which then can be extended to $\R^{m+1}$ by taking $\alpha_{m+1}=0$. 

If $|v_{m+1}^\perp| \leq c$, then $|v_{m+1} -  \sum_{i=1}^m \beta_i v_i| \leq c$; the vector $(-\beta_1,..,-\beta_m,1)$ needs to be normalized, but this lowers the value of $c$ and still gives the result. 

ii) Since $A$ is real and symmetric, it is diagonalizable, thus there exists an orthogonal matrix $Q$ such that 
$Q A Q^{-1}=D=diag[\lambda_1,..,\lambda_n]$ (the diagonal matrix). Assuming that $|\lambda_1| \leq ... \leq |\lambda_n|$ it follows that
$|\lambda_1| \leq \epsilon^\frac1n$. Let $e_1$ be the (unit length) eigenvector corresponding to $\lambda_1$, that is $D e_1 = \lambda_1 e_1$. 
Thus we have $|A Q^{-1} e_1|=|Q A Q^{-1} e_1| = |\lambda_1| \leq \epsilon^\frac1n$. Since $Q^{-1} e_1$ is a unit vectors as well, the conclusion follows.

\end{proof}

In the introduction we have defined a special type of neighborhood for a submanifold;  given a submanifold $S' \subset \R^n$ of codimension $m$, 
we defined $B_\mu(S')$. The main result result of this section connects a localization property of the set of normals to a hypersurface to a localization of the hypersurface in a neighborhood of a submanifold.
The localization of the normals does not automatically imply localization of the hypersurface. Indeed, if the hypersurface is flat, then all normals are equal and no inference can be made. However if the localization comes in a direction "of curvature", then we can deduce localization; morally this is the main idea behind the following result.

\begin{lema} \label{kL}
Let $S$ be a smooth hypersurface obeying \eqref{smooth} and of small diameter. We assume that there exists $\nu_0 >0$ such that
for every $\zeta \in S$ there exists a linear subspace $V_\zeta \subset T_\zeta S$  with the following property:
\begin{equation} \label{Sv}
|S_{N(\zeta)}v \wedge w| \geq \nu_0 |v| \cdot |w|, \quad \forall v \in V^\perp_\zeta, w \in V_\zeta,  
\end{equation}
where $V^\perp_\zeta$ is the normal space to $V_\zeta$ inside $T_\zeta S$. 
Assume that $\calH \subset \R^n$ is a subspace of dimension $n-m$ with the property that
 \begin{equation} \label{ph}
| \pi_{V_\zeta^\perp} h - h| \ll \nu_0 |h|, \quad \forall h \in \calH^\perp,
\end{equation}
where $\pi_{V_\zeta^\perp}$ is the orthogonal projection onto $V_\zeta^\perp$. Given $0 < \mu \ll 1$, 
 we define the following subset of $S$:
\begin{equation} \label{Nlocal}
\tilde S: = \{ \zeta \in S: | \pi_{\calH^\perp} N(\zeta)| \leq \mu  \}.
\end{equation}
Then there exists a submanifold  $\tilde S'$ of $S$ of codimension $m$ and $0 < \tilde \mu \les \mu$ with the following properties: 

i) $\tilde S \subset B_{\tilde \mu}(\tilde S') \cap S$;

ii) the normal plane to $\tilde S'$ is transversal to $V_\zeta$ in the following sense: 
\begin{equation}
|N_\zeta \wedge V_\zeta| \ges 1, 
\end{equation}
for any $\zeta \in \tilde S'$.

\end{lema}

The subspaces $V_\zeta^\perp$ in the above statement should be thought as the natural replacement of $\calH^\perp$ in $T_\zeta S$; we need to work with such a substitute since the result would not be useful if we were to impose the condition $\calH^\perp \subset T_\zeta S, \forall \zeta \in S$.

\begin{proof} We start the argument with an important consequence of \eqref{Sv}. We fix $\zeta \in S$ and let $S_{N(\zeta)} V_\zeta^\perp \subset T_\zeta S$ be the image of of $V_\zeta^\perp$ under $S_{N(\zeta)}$; $S_{N(\zeta)} V_\zeta^\perp$ is itself a subspace. We will make use of two projectors: $\pi_{V_\zeta^\perp}: T_\zeta S \rightarrow V_\zeta^\perp$, the orthogonal projector onto $V_\zeta^\perp$, and  $\pi_\zeta: T_\zeta S \rightarrow S_{N(\zeta)} V_\zeta^\perp$, the orthogonal projection onto $S_{N(\zeta)} V_\zeta^\perp$.

We claim that there exists $\tilde \nu_1 > 0$ such that the following holds true:
\begin{equation} \label{wecon}
| \pi_{\zeta} e | \geq \tilde \nu_1 |e|
\end{equation} 
for any $e \in V_\zeta^\perp$. We argue by contradiction: assume \eqref{wecon} is false. This implies that there exists a unit vector $e \in V_\zeta^\perp$
such that $|e \cdot S_{N(\zeta)} v| < \tilde \nu_1 |S_{N(\zeta)} v|, \forall v \in V_\zeta^\perp$.  From \eqref{volc} we obtain $|e \cdot S_{N(\zeta)} v| < C_1 \tilde \nu_1 |v|, \forall v \in V_\zeta^\perp$. Using the symmetry of $S_{N(\zeta)}$ we obtain $| S_{N(\zeta)} e \cdot  v| <  C_1 \tilde \nu_1 |v|, \forall v \in V_\zeta^\perp$, which implies $|\pi_{V_\zeta^\perp} S_{N(\zeta)} e| < C_1 \tilde \nu_1$. Let $w = S_{N(\zeta)} e- \pi_{V_\zeta^\perp} S_{N(\zeta)} e$. This vector has the following properties:

- $w \in V_\zeta$;

- its size is bounded below by
\[
|w|^2= | S_{N(\zeta)} e|^2- |\pi_{V_\zeta^\perp} S_{N(\zeta)} e|^2 > \nu_0^2 - (C_1 \tilde \nu_1)^2;
\]
- and above by
\[
|w|^2= | S_{N(\zeta)} e|^2- |\pi_{V_\zeta^\perp} S_{N(\zeta)} e|^2 \leq  | S_{N(\zeta)} e|^2 \leq C_1^2;
\]
- and satisfies also
\[
| S_{N(\zeta)} e \wedge w|= |\pi_{V_\zeta^\perp} S_{N(\zeta)} e| \cdot |w| \leq C_1^2 \tilde \nu_1.  
\]
Thus we obtain a contradiction with \eqref{Sv} provided that $C_1^2 \tilde \nu_1 < \tilde \nu_0 (\tilde \nu_0^2 - (C_1 \tilde \nu_1)^2)^\frac12$; obviously this is achievable if $\tilde \nu_1 >0$ is chosen small enough. 

We now return to the proof of the main statements in our theorem. We will be making two simplifying assumptions:
there exists $\zeta_0 \in S$ such that $N(\zeta_0) \in \calH$ and that $\calH^\perp \subset V_{\zeta_0}^\perp$. 
Based on our hypotheses, both these assumption are true up to errors $\ll \nu_0$; if we did not make them, we would have to work with 
$\pi_{V_{\zeta_0}^\perp} e_1,.., \pi_{V_{\zeta_0}^\perp} e_m$ instead in the argument below. 

Let $e_1,..,e_m$ be an orthonormal base in $\calH^\perp$. Since $N_0=N(\zeta_0) \in \calH$ and $e_1,..,e_m$ are all perpendicular to $\calH$, they are perpendicular to $N(\zeta_0)$ and therefore $e_1,..,e_m \in T_{\zeta_0} S$. 

Our first claim is that $(\calN)^{-1}(\calH):= \{ \zeta \in S: N(\zeta) \in \calH \}$ is a submanifold of $S$ of codimension $m$, thus effectively identifying $\tilde S'$.  To do so it is convenient to work in a system of coordinates. 

We let $\pi_0=\pi_{\zeta_0}: T_{\zeta_0} S \rightarrow S_{N(\zeta_0)} V_{\zeta_0}^\perp$ be the orthogonal projection defined above. We first note the following
\begin{equation} \label{pi0e}
|\pi_0 e_i \wedge \calH | \ges 1, \quad \forall i=1,..,m.
\end{equation}
If this were not the case, then, since $e_i \in \calH^\perp$, $|\pi_0 e_i \cdot e_i| \ll 1$ and this would be in contradiction with \eqref{wecon}. 

Now, for every $i \in \{1,..,m \}$ we have that $\pi_0 e_i \in  S_{N(\zeta_0)} V_{\zeta_0}^\perp$, therefore there exists $\tilde e_i \in V_{\zeta_0}^\perp$ with the property
$S_{N(\zeta_0)} \tilde e_i = \pi_0 e_i$. At first $\tilde e_i$ is uniquely determined modulo $ker S_{N(\zeta_0)}$; however 
\eqref{Sv} implies $ker S_{N(\zeta_0)} \cap V_{\zeta_0}^\perp=\emptyset$, and this makes the choice of $\tilde e_i$ unique within $V_{\zeta_0}^\perp$.

Note that it is not necessary that $\{\tilde e_i\}_{i=1,..,m}$ form an orthonormal system, not even an orthogonal one. However they have two crucial properties. 
First, their length is bounded as follows:
\begin{equation} \label{elen}
|\tilde e_i| \leq \tilde \nu_0^{-1}, \quad \forall i=1,..,m.
\end{equation}
Indeed, using \eqref{Sv}, $\tilde \nu_0 |\tilde e_i| \leq |S_{N_{\zeta_0}} \tilde e_i| = |\pi_0 e_i| \leq |e_i|= 1$.

Second, the set of vectors $\{ \tilde e_i\}_{i=1,..,m}$ is transversal in the following sense: there exists $\tilde \nu_2 > 0$ such that 
\begin{equation} \label{etrans}
|\tilde e_1 \wedge .. \wedge \tilde e_m| \geq \tilde \nu_2. 
\end{equation}
 We argue by contradiction here: assume the above is false, that is $|\tilde e_1 \wedge .. \wedge \tilde e_m| < \tilde \nu_2$. From
 \eqref{volc} we obtain $|S_{N(\zeta_0)}\tilde e_1 \wedge .. \wedge S_{N(\zeta_0)} \tilde e_m| <  C_m \tilde \nu_2$, therefore 
 $|\pi_0 e_1 \wedge .. \wedge \pi_0 e_m| \leq C_m \tilde \nu_2$. This further implies
 $|\frac{\pi_0 e_1}{|\pi_0 e_1|} \wedge .. \wedge \frac{\pi_0 e_m}{|\pi_0 e_m|}| \leq \frac{C_m \nu_2}{|\pi_0 e_1| \cdot .. \cdot |\pi_0 e_m|} 
 \leq C_m \tilde \nu_2 \tilde \nu_1^{-m}$, where we invoked \eqref{wecon} in the last inequality. Using Lemma \ref{volc} it follows that there is a unit vector $\alpha \in \R^m$ such that the size of
 $e=\alpha_1 \frac{e_1}{|\pi_0 e_1|} +..+\alpha_k \frac{e_m}{|\pi_0 e_m|}$ satisfies $|\pi_0 e| \leq (C_m \tilde \nu_2)^\frac1m \tilde \nu_1^{-1}$.
  On the other hand, 
 \[
 |e|^2= \sum_{i=1}^m \frac{\alpha_i^2}{|\pi_0 e_i|^2} \geq  1,
 \]
 while from \eqref{wecon} it follows that $|\pi_0 e| \geq \tilde \nu_1$. Thus if $\tilde \nu_2$ is chosen such that $(C_m \tilde \nu_2)^\frac1m \tilde \nu_1^{-1} \leq \tilde \nu_1$, we obtain a contradiction. This establishes \eqref{etrans}. 
  
Since, by \eqref{elen} $|\tilde e_i| \leq \nu_0^{-1}, i=1,..,m$,  \eqref{etrans} also implies that $|\tilde e_i| \geq \tilde \nu_2 \tilde \nu_0^{-(m-1)}, i=1,..,m$. 
Thus we have produced a system of vectors with $|\tilde e_i| \approx 1$ and with $|\tilde e_1 \wedge .. \wedge \tilde e_m| \ges 1$. We can complete
this with a basis $\{\tilde e_i\}_{i=1,..,n-1}$ in $T_{\zeta_0} S$ with similar properties; indeed take $\{e_{m+1},..,e_{n-1}\}$ to be an orthornormal basis
to the $(span(\tilde e_1,..,\tilde e_m))^\perp$, the normal subspace to the subspace spanned by $\tilde e_1,.., \tilde e_m$. 

Since $N(\zeta_0) \perp T_{\zeta_0} S$, we also obtain a basis of vectors in $\R^n$ by adding $N_0=N(\zeta_0)$ to the system $\{\tilde e_i\}_{i=1,..,n-1}$. 
We let $\pi_{N_0}: \R^{n} \rightarrow \R^{n-1}$ be the projector along $N_0$ onto $\R^{n-1}$ which is identified with $T_{\zeta_0} S$ and where
$\{\tilde e_i\}_{i=1,..,n-1}$ forms a basis. We define the parametrization of $S$ by simply projecting $S$ onto $\R^{n-1}$ along $N_0$;  this is a good parametrization in that it obeys \eqref{smooth}. We let $U$ be the image of $S$ under this parametrization, that is $U=\pi_{N_0}(S)$. In $\R^{n-1}$ we use coordinates $(\xi_1,..,\xi_{n-1})$ with the convention that $\xi_{1}, .., \xi_{n-1}$ are the coordinates in the directions of $\tilde e_1,.., \tilde e_{n-1}$ respectively; a apparent drawback of this system of coordinates is that is not Cartesian since $\tilde e_1,.., \tilde e_m$ is not an orthonormal system, but this plays no role in our analysis.  We define 
$\xi_0=\pi_{N_0}(\zeta_0)$. The parametrization of $S$ is now given by $\Sigma: U \rightarrow \R^n$ where $(\pi_{N_0} \circ \Sigma) (\xi)=\xi, \forall \xi \in U$. 
$(\calN)^{-1}(\calH)$ is given by the set of equations
\[
N(\Sigma(\xi))\cdot e_i=0, \quad \forall i=1,..,m. 
\]
We define $F: U \subset \R^{n-1} \rightarrow \R^m, F=(F_1,..,F_m)$ by $F_i(\xi): = N(\Sigma(\xi)) \cdot e_i, \forall i=1,..,m$ 
and aim to solve the equation $F(\xi)=0$. This is done using the implicit function theorem and for this purpose we compute
\[
\frac{D(F_1,..,F_m)}{D(\xi_1,..,\xi_m)} (\xi_0)= \det [S_{N(\Sigma(\xi_0))} \tilde e_i \cdot e_j]_{i,j=1,..,m} = \det [\pi_0 e_i \cdot e_j]_{i,j=1,..,m}.
\]
The matrix $[\pi_0 e_i \cdot e_j]_{i,j=1,..,m}$ is symmetric because $\pi_0$ is a symmetric operator. 
Our goal is to prove that there exists $\tilde \nu_3>0$ such that
\[
|\det [\pi_0 e_i \cdot e_j]_{i,j=1,..,m}| \geq \tilde \nu_3. 
\]
Assume by contradiction that $|\det [\pi_0 e_i \cdot e_j]_{i,j=1,..,m}| < \tilde \nu_3$. 
By Lemma \ref{vols} part ii), there exists $\alpha=(\alpha_1,..,\alpha_m) \in \R^m$ a unit vector such that 
\[
|(\sum_{i=1}^m \alpha_i \pi_0 e_i \cdot e_j)_{j=1,..,m}| \leq \tilde \nu_3^\frac1m.
\]
We let $e=\sum_{i=1}^m \alpha_i  e_i$ be the unit vector in $(\calH)^\perp$ and conclude that 
\[
|(\pi_0 e \cdot e_j)_{j=1,..,m}| < \tilde \nu_3^\frac1m.
\]
From this we further estimate
\[
|\pi_0 e \cdot e| \leq \sum_{i=1}^m |\alpha_i| |\pi_0 e \cdot e_i| \leq |\alpha| |(\pi_0 e \cdot e_j)_{j=1,..,m}| < \tilde \nu_3^\frac1m,
\]
and conclude with $|\pi_0 e|^2 = |\pi_0 e \cdot \pi_0 e|= |\pi_0^2 e \cdot  e|= |\pi_0 e \cdot  e | < \tilde \nu_3^\frac1m$. On the other hand,
from \eqref{wecon} we have $|\pi_0 e|^2 \geq \tilde \nu_1^2 |e|^2=\tilde  \nu_1^2$, thus we obtain a contradiction if $\tilde \nu_3$ is chosen such that
$\tilde \nu_3^\frac1m \leq \tilde \nu_1^2$.  

We can invoke the implicit function theorem to conclude that there exists an open neighborhood $U' \subset U$ of $\xi_0$ and 
$\Sigma': U' \cap \R^{n-m-1} \rightarrow \R^m$ such that
\[
\Sigma^{-1} (\calN)^{-1}(\calH) \cap U'=\{ (\Sigma'_1(\xi_{m+1},..,\xi_{n-1}), .., \Sigma'_m(\xi_{m+1},..,\xi_{n-1}), \xi_{m+1},..,\xi_{n-1}\}.
\]
This gives us the correct structure for $S'=\Sigma^{-1} (\calN)^{-1}(\calH) \cap U'$. 
The size of the neighborhood $U'$ of $\xi_0$ is uniform with respect to the choice of $\xi_0$ since it depends on parameters for which we have uniform bounds across $S$, see \eqref{smooth}, \eqref{Sv} and \eqref{wecon}; to be more precise, the result gives $U'=B_r(\xi_0) \cap U$ and the size of $r$ is uniform with respect to the choice of $\xi_0$. If we assume that the diameter of $S$ is small enough, then $U=\pi_{N_0}(S) \subset B_r(\xi_0)$
and the local solution becomes the global one.

From the above analysis it also follows that $\tilde S'=\Sigma(S')$ is a submanifold of $S$ of codimension $m$.

Next we prove that $\tilde S$ belongs to the correct neighborhood of $\tilde S'$. To do so, it suffices to prove that $\pi_{N_0} \tilde S \subset B_{\tilde \mu} S'$,
for some $\tilde \mu \les \mu$. We rely on the fact that the system $\{ \tilde e_i \}_{i=1,..,m}$ is transversal to $S'$.  We fix $\xi \in S'$, fix $i \in \{1,..,m\}$, and let $\gamma(t)=\xi+t \tilde e_i \in U, |t| \ll 1$; the range of $t$ is constrained by the size of $U$
since we impose $\gamma(t) \in U$. We have
\[
\begin{split}
N(\Sigma(\gamma(t)))-N(\Sigma(\gamma(0))) & = \int_{0}^{t} S_{N(\gamma(s))} d\Sigma(\gamma(s)) \gamma'(s) ds  \\
& = t S_{N(\gamma(0))} d\Sigma(\gamma(0)) \gamma'(0) + O(t^2) \\
&= t S_{N(\zeta_0)} d\Sigma(\xi_0) \tilde e_i + o(t) +O(t^2) \\
&=  t \pi_0 e_i + o(t). 
\end{split}
\]
In the above we used the following straightforward facts:  $d\Sigma(\xi_0)=Id$, $\|S_{N(\gamma(0))}-S_{N(\zeta_0)} \| \ll 1$ and 
$\|d\Sigma(\gamma(0))-d\Sigma(\xi_0)\| \ll 1$; the last two inequalities involve matrix norms and they are based on the small diameter property of $S$. 

From \eqref{wecon} and \eqref{pi0e}, we have that $|\pi_0 e_i| \ges 1$ and $| \pi_0 e_i \wedge \calH| \ges 1$; since $N(\Sigma(\gamma(0))) \in \calH$,
it follows that a necessary condition of $\Sigma(\gamma(t)) \in \tilde S$ is that $|t| \leq C \mu $. 
Now the system $\{ \tilde e_i \}_{i=1,..,m}$ is transversal to $S'$, whose co-dimension is $m$, therefore $\pi_{N_0} \tilde S \subset B_{\tilde \mu} S'$,
for some $\tilde \mu \les \mu$.

Let us look into ii). Any unit vector $n \in N_\zeta$ has the property that its projection on the $span\{ \tilde e_1,..,\tilde e_m \} \subset V_{\zeta_0}^\perp$ 
has length bounded from below by a universal constant; this simply follows from the structure of $\tilde S'$ and the smallness of the diameter of all submanifolds involved. As a consequence $|n \wedge V_{\zeta_0}^\perp| \ges 1$, thus $|n \wedge V_{\zeta}^\perp| \ges 1$. This implies that $|N_\zeta \wedge V_\zeta| \ges 1$.

\end{proof} 

\section{Wave packets} \label{swp}

In this section we introduce wave packet theory for operators of type $\calE$. We use the setup from our previous works, see \cite{Be2} for more details,
a setup that was inspired by the work of Tao \cite{Tao-BW}. We make specific choices for the parameters used so as to (almost) match the construction that Guth uses in \cite{Gu-II}. 

Consider the extension operator $\calE f$ associated to the hypersurface $S$ that is parametrized by $\Sigma: U \subset \R^{n-1} \rightarrow \R^n$. 
Without restricting the generality of the argument, we assume that the parametrization is given by $\Sigma(\xi)= (\xi,\varphi(\xi))$. 

We choose a parameter $0 < \delta \ll 1$ to be specified later. Given the ball $B_R(0)$ of radius $R$ and centered at $0$, we introduce the wave packet decomposition for $\calE f$ adapted to it - in fact what follows below is the wave packet decomposition adapted to the region $|x_n| \leq R$.
We make use of two lattices:  $\mathcal{L}=R^{-\frac12} \Z^{n-1} \cap U$ and let $L$ be the lattice $L= R^{\frac{1+\delta}2} \Z^{n-1}$. Both these lattices are build along the directions of the standard orthonormal system of vectors $e_1,..,e_{n-1}$ in $\R^{n-1}$. 

We decompose $U$ as follows: 
\[
U= \bigcup_{\xi_0 \in \mathcal{L}} A_{\xi_0}
\]
where $A_{\xi_0}$ consists of the points in $U$ that are closer to $\xi_0$ than any other elements of $\mathcal{L}$. 
Therefore $A_{\xi_0}$ belongs to the $O(R^{-\frac12})$ neighborhood of $\xi_0$.  We have the following decomposition of $f$:
\[
f=\sum_{\xi_0 \in \mathcal{L}} \chi_{A_{\xi_0}} f. 
\]
The terms above have good frequency support and next we proceed with the spatial localization. Let $\eta_1:\R \rightarrow [0,+\infty)$ be a Schwartz function, normalized in $L^1$, that is $\| \eta_1 \|_{L^1}=1$, and with Fourier transform supported in $(-1,1)$. Given some $r>0$ we denote by $\eta_r(x)=r^{-1} \eta_1(r^{-1} x)$ and note that $\hat \eta_r$ is supported in $(-r^{-1},r^{-1})$. Based on this function, we can construct similar ones acting on any $\R^{n}$ space and with similar properties; we will abuse notation and use the same $\eta_1$ notation. Thus we define $\eta_1(x_1,..,x_n)=\Pi_{i=1}^n \eta_1(x_i)$. As above we define $\eta_r(x)=r^{-n} \eta_1(r^{-1} x)$ and note that $\hat \eta_r$ is supported in the $(-r^{-1},r^{-1})^{n}$. 

For each $x_0 \in L$, 
\[
\eta^{x_0}_{R^\frac{1+\delta}2}(x) = \eta_{R^\frac{1+\delta}2} (x-x_0)
\]
and notice that, by the Poisson summation formula and properties of $\eta_0$, 
\begin{equation} \label{pois}
\sum_{x_0 \in L} \eta^{x_0}_{R^\frac{1+\delta}2} =1.
\end{equation}
We define $f_{x_0,\xi_0}= \hat \eta^{x_0}_{R^\frac{1+\delta}2}  \ast \chi_{A_{\xi_0}} f$; from the above considerations we obtain the following decomposition of $f$:
\[
f=\sum_{x_0 \in L, \xi_0 \in \mathcal{L}} f_{x_0,\xi_0}. 
\]
With $x_T \in L, \xi_T \in \mathcal{L}$ we define the tube $T:=\{ x=(x',x_n) \in \R^n: |x'-x_T + \nabla \varphi(\xi_T) x_n| \leq  R^{\frac12+\delta}, |x_n| \leq R \}$ and denote by $\mathcal{T}$ the set of such tubes. For each $T$ as above, we define $v(T)=\frac1{\sqrt{1+|\nabla \varphi(\xi_T)|^2}}(-\nabla \varphi(\xi_T),1)$, the unit vector describing the direction of the center line of $T$. We let $f_T=f_{x_T,\xi_T}$ and recall that we have the decomposition
\begin{equation} \label{lind}
f= \sum_{T \in \mathcal{T}} f_{T}.
\end{equation}
The relevance of the tubes is that $\calE f_T$ is essentially localized within the tube $T$ (in fact in a thinner tube $T':=\{ x=(x',x_n) \in \R^n : |x'-x_T + \nabla \varphi(\xi_T) x_n| \leq  R^{\frac{1+\delta}2}, |x_n| \leq R \}$ of width $R^{\frac{1+\delta}2}$) and decays very fast outside $T$, that is
\[
\| \calE f_T\|_{L^\infty(T^c)} = O(R^{-N}). 
\]
This is a consequence of the more general inequality:
\begin{equation} \label{decay}
| \calE f_T(x)|  \ls R^{-\frac12} (1+\frac{d(x,T')}{R^{\frac{1+\delta}2}})^{-N} \|f_T\|_{L^2},
\end{equation}  
whose proof can be found in \cite{Be2}. We also recall the standard orthogonality property that this decomposition has: 
if $\calT' \subset \calT$ the following holds true:
\begin{equation} \label{sel}
\|  \sum_{T \in \calT'} f_T  \|^2_{L^2} \les   \sum_{T \in \calT'} \| f_T  \|^2_{L^2}. 
\end{equation}
We recall that all the above properties are meant to be read in the spatial regime $|x_n| \leq R$ of which $B_R(0)$ is a subset of. 
With little modifications this would work well for any $B_R(p)$ with $|p_n| \ll R$, but it needs to be modified if we need to use it on 
$B_R(p)$ with $|p_n| \ll R$ with $|p_n| \ges R$. The solution is very simple - we essentially have to recenter the decomposition so as to have $p_n=0$.
This is done as follows: keeping in mind that $\calE f (\cdot, x_n)$ is an $L^2$-isometry with respect to $x_n$, we simply
redo the above construction for $\calE f (\cdot, p_n)$ and translate all the entities by $(0,p_n)$. This is what we will refer to as the wave packet decomposition adapted to $B_R(p)$; we also denote by $\calT (B_R(p))$ the resulting family of wave packets.

\section{Tools and results from algebraic geometry} \label{TAG}
In this section we collect all the tools from algebraic geometry that we employ in this paper. 
The algebraic geometry setup here follows closely the one developed by Guth in \cite{Gu-I,Gu-II} - for more details on our exposure 
we refer the reader to \cite{Gu-II}.  

The first concept that we introduce is that of an real algebraic variety: this is the locus of common zeros of a collection of polynomials. 
More precisely, an algebraic variety
takes the form $Z(P_1,...,P_d)=\{ x \in \R^n: P_i(x)=0, \forall i=1,..,d\}$; here $P_i \in \R[X_1,..,X_n]$ are real valued polynomials. The algebraic varities do not necessarily have good geometric properties, in particular they may lack a good manifold structure. To avoid this problem, a more restrictive class is introduced. The variety $Z(P_1,...,P_{m})$ is said to be a transverse complete intersection if
\begin{equation} \label{ci}
\nabla P_1(x) \wedge ... \wedge \nabla P_{m}(x) \neq 0, \qquad \forall x \in Z(P_1,...,P_{m})
\end{equation}
A transverse complete intersection $Z(P_1,...,P_{m})$ is a smooth $n-m$-dimensional manifold. In particular, if $n=m$, we have the following result. 
\begin{theo} [Theorem 5.8, \cite{Gu-II}/Theorem 5.2 \cite{CKW}] \label{ftci}
A transverse complete intersection $Z(P_1,...,P_{n})$ is a finite set of cardinality at most  $\Pi_{i=1}^n deg P_i$. 
\end{theo}

From Sard's theorem it follows that
 generically the set of common zeroes of a family of polynomials is a transverse complete intersection; the precise statement is made below. 

\begin{lema}[Lemma 5.1, \cite{Gu-II}] \label{lci}
If P is a polynomial on $\R^n$, then for almost every $c_0 \in \R$, $Z(P+c_0)$ is a transverse complete intersection.
More generally, suppose that $Z(P_1,...,P_{m})$ is a transverse complete intersection and that $P$ is another polynomial. Then for almost every $c_0 \in \R$, 
$Z(P_1,...,P_{m},P+c_0)$ is a complete transverse intersection. 
 \end{lema}

The most important result in this section  is the following theorem about polynomial partitioning.
\begin{theo}[Theorem 5.5, \cite{Gu-II}] \label{PP}
Suppose that $W \geq 0$ is a non-zero $L^1$ function on $\R^n$. Then, for any degree $D $ the following holds:

There is a sequence of polynomials $Q_1,...,Q_S$ with the following properties. We have $\sum_{s=1}^S deg Q_s \les D$ and $2^S \approx D^n$.
Let $P=\Pi_{s=1}^S \tilde Q_s= \Pi_{s=1}^S (Q_s+c_s)$ where $c_s \in \R$. Let $O_i$ be the open sets given by the sign conditions of $\tilde Q_s$, that is
\[
O_i=\{ x \in \R^n| Sign \ \tilde Q_s(x)=\sigma_s \}
\]
where $\sigma_s \in \{-1,+1\}$. There are $2^S \approx D^n$ cells $O_i$ and $\R^n \setminus Z(P)=\bigcup_i O_i$. 

If the constants $c_s$ are sufficiently small, then for every $O_i$ 
\[
\int_{O_i} W \approx D^{-n} \int_{\R^n} W.
\]
\end{theo}
From Lemma \ref{lci} it follows that for almost every $c_s$, $Z(\tilde Q_s)$ is a complete intersection for each $s$, which in turn implies that $Z(P)$ is generically a finite union of transverse complete intersections. More generally, if $Z(P_1,...,P_{m})$ is a transverse complete intersection and $m \leq n-1$,
then for a generic choice of constants $c_s$, $Z(P_1,...,P_{m}, \tilde Q_s)$ is a transverse complete intersection for every $s$. 

The next result is helpful in controlling the tangent plane of a variety. Here we assume that the $n-m$-dimensional variety $Z=Z(P_1,...,P_{m})$ is an  transverse complete intersection and that $m \leq n-1$. Given $w \in \Lambda^{n-m} \R^n$, we define $Z_w$ by
\begin{equation} \label{Zw}
Z_w: =\{ x \in Z | \nabla P_1(x) \wedge ... \wedge \nabla P_{m}(x) \wedge w = 0\}. 
\end{equation}
The expression $g_w(x):= \nabla P_1(x) \wedge ... \wedge \nabla P_{m}(x) \wedge w$ is a polynomial with degree at most $Deg P_1 + ...+Deg P_{m}$, therefore $Z_w$ is an algebraic variety. The following result states that, generically, $Z_w$  is a a smooth transverse complete intersection.

\begin{lema}[Lemma 5.6, \cite{Gu-II}] \label{Lw}
For almost every $w \in \Lambda^{n-m} \R^n$, $Z_w= Z(P_1,...,P_{m},g_w)$ is a smooth transverse complete intersection. 
\end{lema}

The last result we need from \cite{Gu-II} provides a useful tool for controlling the transverse intersections between a tube and a variety.
Given a direction $v_0 \in \R^n$ with $|v_0|=1$, a radius $r$ and a point $x_0 \in \R^n$, we let $T(x_0,v_0,r)=\{x_0+tv_0+e: t \in \R, e \in \R^n, |e| \leq r \}$
be the tube that passes through $x_0$, has direction $v_0$ and has radius $r$.

Suppose $Z \subset \R^n$ is a transverse complete intersection; we define the following two sets:
\[
Z_{v_0, \leq \alpha}= \{z \in Z: \angle(v_0, T_z Z) \leq \alpha\}, \qquad  Z_{v_0, > \alpha} = Z \setminus Z_{v_0, \leq \alpha}.
\]
We have the following result. 
\begin{lema}[Lemma 5.7, \cite{Gu-II}] \label{Zan}
Suppose that $Z=Z(P_1,...,P_{m})$ is a transverse complete intersection and that the polynomials $P_i, i=1,..,m$ have degrees at most $D$. Let 
$T=T(x_0,v_0,r)$ be a tube as above. Then for any $\alpha >0$, $Z_{v_0,>\alpha} \cap T$ is contained in a union of $\les D^n$ balls of radius $\ges r\alpha^{-1}$.
\end{lema}

\subsection{Tangent and non-tangent tubes to an algebraic variety} \label{sat}

In this section we assume that $Z=Z(P_1,..,P_{n-m})$ is a transverse complete intersection. 
We are given three parameters, sizes $R$ and $\rho$, and angle $\alpha$, satisfying the condition $R^{-\frac12+\delta} \ll \alpha \ll 1$. We assume we have a family of wave packets adapted to the ball $B_R$. Given a ball $B \subset B_R$ of size $\rho$, first we select the tubes that intersect $N_{\alpha R}(Z) \cap B$, that is we consider tubes $T$ such that $T \cap N_{\alpha R}(Z) \cap B \neq \emptyset$. 
These tubes are divided into two classes: tangential and non-tangential. A tube is called tangential if for any $x \in T$ and $y \in Z \cap 2B$ with 
$|x-y| \leq 2 \alpha R$,
\begin{equation} \label{tangc2}
\angle(T,T_y Z) \leq \alpha.
\end{equation}

A tube is called non-tangential if it is not tangential, that is there exists $x \in T$ and $y \in Z \cap 2B$ with $|x-y| \leq 2 \alpha R$ such that $\angle(T,T_y Z) > \alpha$. 
This implies 
\begin{equation}
N_{2 \alpha R} (T) \cap 2B \cap Z_{v(T), > \alpha} \neq \emptyset. 
\end{equation}

In \cite{Gu-II} the definition of tangent tubes has the additional property that $T \cap B_R \subset N_{2\alpha R} (Z) \cap B_R$. While, in that particular context, it is not clear to us the consistency of properties listed for tangential and transverse tubes, such a property can be recovered along the lines of a similar argument used by Guth in the proof of Lemma 4.9 in \cite{Gu-I}. 

\section{Main Argument} \label{sma}

The main result of this paper follows from the following

\begin{prop} \label{indp}
Assume the setup in Theorem \ref{mainT}.
 Then for any $\epsilon >0$, the following holds true: for any $R >0$,
\begin{equation}  \label{mrei-ind}
\| \Pi_{i=1}^k \calE_i f_i \|_{L^p(B_R)} \lesssim C(\epsilon,p) R^{ \epsilon} 
 \Pi_{i=1}^k \| f_{i} \|_{L^2(U_i)}.  
\end{equation}
\end{prop}

Indeed, by invoking Proposition \ref{epsr} we can remove the $\epsilon$-loss above in the regime $p>\frac{2(n+k)}{k(n+k-2)}$ 
(notice that $\frac{2(n+k)}{k(n+k-2)} > \frac2{k-1}$) and obtain the claim in Theorem \ref{mainT}.

We proceed with the proof of Proposition \ref{indp}. The main idea is the use of an induction on scale type argument that involves the polynomial partitioning along the lines developed by Guth in \cite{Gu-II}. 

We let $A(R)$ be the best constant for which the following inequality
\begin{equation}  \label{mrei-ii}
\| \Pi_{i=1}^k \calE_i f_i \|_{L^p(B_R)} \leq A(R) \Pi_{i=1}^k \| f_{i} \|_{L^2(U_i)},
\end{equation}
holds true for any ball $B_R \subset \R^n$ of radius $R$. 

There is one important concept that is missing from the above definition. The argument will be based on induction on scales; this requires various localizations  
on the physical side and that impacts the localization on the Fourier side. This motivates the use of the margin concept which we introduce below. 
We assume that, for each $i$, we are given a reference set $V_i$ inside which we want to keep all functions supported. 
If $f_i$ is supported in $U_i \subset V_i$ we define the margin of $f_i$ relative to $V_i$ by
\[
\mbox{margin}^i (f_i) := \mbox{dist}(\mbox{supp} (f_i), V_i^c). 
\]
The idea is that instead of keeping $U_i$ fixed as stated in \eqref{mrei-ii}, we allow it to vary in the following sense:  
we let $A(R)$ be the best constant for which the following inequality
\begin{equation} 
\| \Pi_{i=1}^k \calE_i f_i \|_{L^p(B_R)} \leq A(R) \Pi_{i=1}^k \| f_{i} \|_{L^2},
\end{equation}
holds true for any ball $B_R \subset \R^n$ of radius $R$ and any $f_i$ whose margin satisfies $\mbox{margin}^i(f_i) \geq m-R^{-\frac14}$; here $0 < m \ll 1$ 
is meant to be a very small constant. In the argument below we ignore the margin concept so as to focus our attention on the essential parts of the argument;
 in Section \ref{concl}, where we put together all our estimates to obtain the final bound on $A(R)$ above, we will explain how the margin concept can be easily inserted back into the full argument. 

For a Lebesgue measurable set $A$ we define the measure 
\[
\mu(A)= \int_C  |\Pi_{i=1}^k \calE_i f_i (x) |^p dx.
\]
When we work within a specific ball $B_R$ we employ the wave packet decomposition adapted to it, that is with $\calT(B_R)$; however to keep notation simple, this will be implicit in the exposition and we simple use the notation $\calT$ for the set of wave packets used. 
Let $C \subset B_R$ be a Lebesgue measurable set and define
\[
f_{i,C}= \sum_{T \in \calT_i: T \cap C \neq \emptyset} f_{i,T}.
\]
It is an easy exercise to prove the following
\begin{equation} \label{Cloc}
\| \Pi_{i=1}^k \calE_i f_i \|_{L^p(C)} \leq A(R) \Pi_{i=1}^k \| f_{i,C} \|_{L^2(U_i)} + O(R^{-N}) \|f_i\|_{L^2(U_i)}. 
\end{equation}
We proceed with the polynomial partitioning argument. For clarity purposes, we present the starting point in the following section and then consider the more general setup 
in the succeeding section. 

\subsection{The polynomial partitioning - I} 

We apply Theorem \ref{PP} for the function $ |\Pi_{i=1}^k \calE_i f_i (x) |^p \cdot \chi_{B_R}(x)$; this produces a polynomial $P$ of degree at most $D_1$ and with all the properties provided there. For now the important properties are
$\R^n \setminus Z(P)=\bigcup_j O_j$, the cardinality of the set of $j$'s is $\approx D_1^n$ and $\mu(O_j) \approx D^{-n}_1 \mu (B_R)$.

Let us assume that the following holds true:
\begin{equation} \label{nalg1}
\mu(N_{R^{\frac12+ \delta_0}}(Z) \cap B_R) \ll D_1^{-n} \mu(B_R),
\end{equation}
which is referred to as the non-algebraic case since the measure is concentrated away from the algebraic variety. Here $\delta_0$
is a small parameter with the property $ \delta \ll \delta_0 \ll 1$. 
If we let $O_j'=O_j \setminus N_{R^{\frac12+ \delta_0}}(Z) $, the following holds true
\[
\mu(O_j') \approx D_1^{-n} \mu(B_R),
\]
for all $j$'s. For each $f_i$ and each cell $O_j'$, we collect the relevant tubes from $f_i$ that go through $O_j'$ by defining
\[
f_{i,j} = \sum_{T \in \calT_i: T \cap O_j' \neq \emptyset} f_{T}. 
\]

In \cite{Gu-I}, see Lemma 3.2, it is shown that for every $i$ and for any tube $T \in \calT_i$ of radius $R^{\frac12+\delta}$, T enters at most $D_1+1$ cells $O_j'$.
Therefore for each $i$, 
\begin{equation} \label{encount}
\sum_j \| f_{i,j} \|^2_{L^2} \lesssim D_1 \| f_{i} \|^2_{L^2}.
\end{equation}
 Since there are $\approx D_1^{n}$ cells $O_j'$ it follows that, for most cells
 \[
 \| f_{i,j} \|_{L^2}^2 \lesssim D_1^{1-n} \| f_{i} \|^2_{L^2}. 
 \]
 Assuming $D_1$ is large enough with respect to $k$, it follows that, for most cells the above inequality holds true for every $i$ as well. 
 We also make the observation that $D_1$ and later choices of degrees will be absolute constants independent of $R$. 
 By choosing such a cell $O_j'$ the following holds true
  \[
  \begin{split}
 \mu(B_R) \approx D_1^{n} \mu(O_j') & \lesssim D_1^{n}   \left( A(R)  \Pi_{i=1}^k \| f_{i,j} \|_{L^2} + O(R^{-N}) \Pi_{i=1}^k \| f_{i} \|_{L^2(U_i)}\right)^p \\
 & \lesssim \left ( A(R)^p D_1^{n} D_1^{\frac{(1-n)kp}2} + O(R^{-Np}) \right)  \Pi_{i=1}^k \| f_{i} \|_{L^2}^p. 
 \end{split}
 \]
 If the power of $D_1$ is negative, we can choose $D_1$ large enough so that $C D_1^{n+\frac{(1-n)kp}2} < 1$ ($C$ is the constant used in the definition of $\les$) and obtain a satisfactory estimate for this case:
 \begin{equation} \label{gest1}
  \mu(B_R)^\frac1p \leq (\frac12 A(R) + O(R^{-N})) \Pi_{i=1}^k \| f_{i} \|_{L^2}. 
 \end{equation}
 For this to happen we need  $p=\frac{2(n+k)}{k(n+k-2)}>\frac{2 n}{k(n-1)}$ which holds true given that $n > k$. 
 
In running the argument in the previous subsection we heavily relied on the underlying hypothesis 
\eqref{nalg1}. We now assume we are in the opposite scenario, that is
\begin{equation} \label{alg1}
\mu(N_{R^{\frac12+ \delta_0}}(Z) \cap B_R) \gtrsim D_1^{-n} \mu(B_R),
\end{equation} 
which is referred to as the algebraic case. From the structural properties of Theorem \ref{PP}, it follows that there exists a polynomial $P_1$, occurring in the factorization of $P$, such that $deg P_1 \leq D_1$, $Z(P_1)$ is a transverse complete intersection and
\begin{equation} \label{alg2}
\mu(N_{R^{\frac12+ \delta_0}}(Z(P_1)) \cap B_R) \gtrsim \frac{\mu(B_R)}{D_1^n \ln{D_1}}.
\end{equation}

\subsection{Polynomial partitioning - II} \label{pp2}

We assume that we are in a more general setup than the one we just concluded the previous subsection with and quantified in \eqref{alg2}. We introduce a few more concepts that are needed for this purpose. 

We work with a more general algebraic variety $Z=Z(P_1,..,P_{m})$ which is a transverse complete intersection. The polynomials $P_i$'s have degrees less than $D_i$, with $D_1 \leq D_2 \leq .. \leq D_m$. The only natural restriction we impose on $m$ is that $1 \leq m \leq n$. In fact the case $m=n$
is a bit special and easier and will be treated in Section \ref{special}, thus we will assume that $1 \leq m \leq n-1$.

We assume that the following holds true:
\begin{equation} \label{alg4}
\mu(N_{R^{\frac12+ \delta_0}}(Z)  \cap B_R) \gtrsim \frac{\mu(B_R)}{D_m^{\kappa_m}},
\end{equation} 
for some $\kappa_m \in \N$.

Our goal is to reduce the problem to one of the the following three scenarios: 

i) provide a good estimate as in the previous subsection;

ii) further reduce the dimensionality of the algebraic variety where the measure is concentrated (in the sense of \eqref{alg4});

iii) there is a third alternative that needs more preparation to describe. 

First we summarize the steps in \cite{Gu-II} that are necessary to set up a refined version of the polynomial partition; we aim here for a self-contained presentation, and the reader is referred to \cite{Gu-II} for additional details. 
For setting up the polynomial partition, it is necessary to locate a relevant portion of $N_{R^{\frac12+\delta_0}} Z$ where the tangent space at $Z$ is close to a specific choice.  We pick $0 < \gamma_0 \ll 1$, whose smallness depends on parameters such as $\nu$, but not $R$ or $\delta$'s. We say that  $B_{R^{\frac12+\delta_0}}(x_0) \subset N_{R^{\frac12+\delta_0}} Z$ is a regular ball if, on each connected component of $Z \cap B_{R^{\frac12+\delta_0}}(x_0)$, the tangent space  $TZ$ is constant up to an angle $\gamma_0$. In Section \ref{TAG} we introduced, for each $w \in \Lambda^{n- m} \R^n$, the algebraic variety $Z_w \subset Z$ which, for generic choice of $w$, is a transverse complete intersection defined using polynomials of degree $\leq D_1+...+ D_{ m}$. Therefore we can choose a set of $w$'s of cardinality $\lesssim 1$ so that on each connected component of $Z \setminus \bigcup_w Z_w$ the tangent space $TZ$ is constant up to an angle $\gamma_0$.

Assume that the measure is concentrated  on $N_{R^{\frac12+\delta_0}} (\cup_w Z_w) \cap B_R$, that is 
\[
\mu(N_{R^{\frac12+ \delta_0}}( \cup_w Z_w)   \cap B_R) \gtrsim \frac{\mu(B_R)}{D_m^{\kappa_m}}.
\]
Then there exists a $w$ such that the measure is concentrated on $N_{R^{\frac12+\delta_0}} (Z_w)  \cap B_R$ and then we achieved our goal: 
by Lemma \ref{Lw}, $Z_w=Z(P_1,..,P_{ m}, g_w)$ is a complete transverse intersection, thus by setting $P_{m+1}= g_w= \nabla P_1 \wedge ... \wedge \nabla P_{m}\wedge w$ we have reduced the dimensionality of the algebraic variety whose neighborhood contains a relevant portion of $\mu(B_R)$.

Assume now that  the measure is not concentrated on $N_{R^{\frac12+\delta_0}} (\cup_w Z_w)  \cap B_R$, in which case 
\[
\mu(N_{R^{\frac12+ \delta_0}}(Z \setminus \cup_w Z_w)  \cap B_R) \gtrsim \frac{\mu(B_R)}{D_m^{\kappa_m}}.
\]
For a regular ball $B=B_{R^{\frac12+\delta_0}}(x_0)$ we pick a point $z \in B \cap Z$ and define $V_B$ to be the $n- m$ tangent plane $T_z Z$. For an $n- m$ plane $V$, we let $\mathcal{B}_V$ be the set of regular balls $B$ satisfying $\angle (V_B,V) \leq \gamma_0$. With a set of $V$'s of cardinality $\lesssim 1$ we make sure each regular ball belongs to some $\mathcal{B}_V$, therefore there is an $n- m$ plane $V$ such that 
\[
\mu(\bigcup_{B \in \mathcal{B}_V} B ) \gtrsim \frac{\mu(B_R)}{D_m^{\kappa_m}}.
\]

Let $N_{m}=(\bigcup_{B \in \mathcal{B}_{V_m}} B) $ and let $\mu_{m}$ be the restriction of $\mu$ to $N_{m}$; this has the property that $\mu_{m}(N_{m}) \gtrsim \frac{\mu(B_R)}{D_m^{\kappa_m}}$.

Now that we froze the angle of $TZ$, we can introduce the transversal waves. Let $V_{\geq \gamma}=\{ v \in \R^n: \angle (v,V) > \gamma\}$ and 
$V_{\leq  \gamma}=\{ v \in \R^n: \angle (v,V) \leq \gamma\}$. 
 We split each $f_i$ as follows
\[
f_i = f_{i,ntrans} + f_{i,trans},
\]
where
\[
f_{i,trans}=\sum_{T \in \calT_i: v(T) \in V_{> 4 \gamma_0}} f_{i,T}, \qquad f_{i,ntrans}= \sum_{T \in \calT_i: v(T) \in V_{\leq  4 \gamma_0}} f_{i,T}. 
\]
Then we have 
\[
\| \Pi_{i=1}^k \calE_i f_{i} \|_{L^p(N_{m})}^p \lesssim  \sum_{car} \| \Pi_{i=1}^k \calE_i f_{i,car_i} \|_{L^p( N_{m})}^p 
\]
where $\sum_{car}$ runs over all possible combinations $car=(car(1),..,car(k)) \in \{ntrans,trans\}^k$. 

One of these combinations on the right-hand side is dominant, that is there is a choice $car \in \{ntrans,trans\}^k$ such that
\begin{equation} \label{car1}
 \| \Pi_{i=1}^k \calE_i f_{i} \|_{L^p(N_{m})}^p \lesssim  \| \Pi_{i=1}^l \calE_i f_{i,car_i} \|_{L^p( N_{m})}^p, 
\end{equation}
and it suffices to estimate this combination. 

Assume that in such a combination, we have $car_i=trans$ for at least $m$ waves; without restricting the generality of the argument we can assume $car_i=trans$ for $1 \leq i \leq m$. The rest of this subsection will analyze this case and the conclusion will be the following: we can still apply the polynomial partitioning and either be able to use induction on scales or further reduce the setup to where the measure is concentrated in the neighborhood $N_{R^{\frac12+\delta}} Z$ where $Z=Z(P_1,..,P_{m}, P_{m+1})$ is a complete transverse intersection. 

Next we proceed with the polynomial partition argument. We let $\pi: \R^n \rightarrow V$ be the orthogonal projection on $V$ and we let $P_V : V \rightarrow \R$ 
denote a polynomial defined on $V$. We apply the Theorem \ref{PP} to the push-forward measure $\pi_* \mu_m$ on $V$, using the degree $D_{ m +1}$. 
This gives us a polynomial $P_V$ of degree at most $D_{ m +1}$ so that $V \setminus Z(P_V)= \bigcup_j O_{V,j}$, the number of cells $O_{V,j}$ is 
$\approx D_{m+1}^{(n-m)}$ and for each cell $\pi_*(O_{V,j}) \approx D_{m+1}^{-(n- m)} \mu_m(V)$. We also get the structural property $P_V = \Pi_l Q_{V,l}$ where each $Q_{V,l}$ can be perturbed with small constant terms that can be used for transversality purpose. 

Next we extend all the above entities from $V$ to $\R^n$.  $P_V$ is then extended to a polynomial $P$ on $\R^n$ by setting $P(x):= P_V(\pi(x))$. This has the property that $Z(P)=\pi^{-1}(Z(P_V))$. We let $O_j :=\pi^{-1} (O_{V,j})$ and note that $\R^n \setminus Z(P)=\bigcup_j O_j$ and that $\mu_m(O_j) = \pi_* \mu_m (O_{V,j}) \approx D_{ m +1}^{-(n- m)} \mu_m(N_m)$. Similarly, we define $Q_{l}(x)=Q_{V,l}(\pi(x))$ so that $P=\Pi_l Q_l$. Using the fact that each $Q_l$ can be perturbed by small constants, 
by invoking Lemma \ref{lci} for each $l$, $Y_l=Z(P_1,..,P_{ m}, Q_l)$ is a transverse complete intersection. 

We define $W:=N_{R^{\frac12+\delta_0}} Z$ and $O_j':=O_j \setminus W$. The following holds true
\[
W \cap N_m \subset \bigcup_l N_{20 R^{\frac12+\delta_0}}(Y_l). 
\]
The proof is short and can be found in \cite{Gu-II}. 

If $\mu_m(N_m)$ is concentrated on $W \cap N_1$, that is $\mu_m(W \cap N_1) \ges \frac{\mu(B_R)}{D_m^{\kappa_m}}$,
 then it is concentrated in one of the $N_{20 R^{\frac12+\delta_0}}(Y_l)$, for some $l$, thus since $Y_l=Z(P_1,..,P_{ m}, Q_l)$ is a transverse complete intersection, we have reduced the degree of the relevant algebraic variety. 

The other case is when $\mu$ is not concentrated on $W \cap N_m$, that is $\mu_m(W \cap N_1) \ll \frac{\mu(B_R)}{D_m^{\kappa_m}}$. 
Here the argument is similar to the one we employed in the previous 
subsection but with adjusted numerology.

We recall that we have $ \mu_m(O_j') \approx D_{m+1}^{-(n- m)} \mu_m(N_m)$. Next we define 
\[
f_{i,j} = \sum_{T \in \calT_i: T \cap O'_j \neq \emptyset} f_{T}.
\]
Since a tube cannot intersect more than $\approx D_{m+1}$ cells, it follows that for each $i \in \{1,..,k\}$ the following holds true
\begin{equation} \label{encount2}
\sum_j \| f_{i,j} \|^2_{L^2} \les D_{m +1} \| f_i \|^2_{L^2}. 
\end{equation}
 Since there are $\approx D_{m+1}^{n- m}$ cells $O_j'$ it follows that, for most $j$'s
 \begin{equation} \label{gb}
 \| f_{i,j} \|_{L^2}^2 \les D_{ m + 1}^{1-(n- m)} \| f_j \|^2_{L^2}. 
 \end{equation}
However the above analysis improves when we run it for the transversal waves, that is for the waves with $car_i=trans$. 
Here we refine a little the definition of $f_{i,j}$ in the following sense
\[
f_{i,j} = \sum_{T \in \calT_i: T \cap O'_j \cap N_m \neq \emptyset} f_{T}.
\] 
This could have been done for all other $i$'s since we have already restricted our attention to $N_m$ (recall \eqref{Cloc}). 
 From Lemma \ref{Zan} it follows that each tube $T$ has the property that $T \cap Z$ is contained in at most $D_{ m}$ balls of radius $R^{\frac12+\delta_0}$. 
 A tube $T \subset \calT_i$ may enter at most $D_{ m+1}$ cells $O_j'$; however the intersection $T \cap O'_j \cap N_m \neq \emptyset$ for at most 
 $D_{ m}$ cells $O_j'$ - simply because each such intersection will count towards an intersection of $T \cap Z$.

Therefore, for these $i$'s, the estimate \eqref{encount} improves to
\[
\sum_j \| f_{i,j} \|^2_{L^2} \les D_{m} \| f_i \|^2_{L^2},
\]
and further, since there are $\approx D_{ m+1}^{n- m}$ cells $O_j'$ it follows that, for most $j$'s
 \begin{equation} \label{bgb}
 \| f_{i,j} \|_{L^2}^2 \les D_{ m} D_{ m+1}^{-(n- m)} \| f_i \|^2_{L^2}. 
 \end{equation}
To wrap things up, for the transversal waves ($i=1,..,m$) we make use of the improved bound\eqref{bgb}, while for the other one we use the standard bound \eqref{gb}. Then we continue the argument with the following estimates
\[
\begin{split}
 \mu_m(N_m) & \approx D_{ m+1}^{n- m} \mu(O_j') \les D_{ m+1}^{n- m}  \left( A(R)  \Pi_{i=1}^k \| f_{i,j} \|_{L^2} + O(R^{-N}) \Pi_{i=1}^k \| f_{i} \|_{L^2(U_i)}\right)^p \\
 & \les \left( A(R) D_{ m+1}^{n- m} D_{ m+1}^{\frac{(1-(n- m))(k-m)p}2} (D_{ m} D_{ m+1}^{-(n- m)})^{\frac{mp}2} + O(R^{-Np}) \right)  \Pi_{j=1}^k \| f_{j} \|_{L^2}^p
   \\
 &  = \left( A(R) D_{ m+1}^{(n- m)(1-\frac{kp}2) +\frac{(k-m)p}2}  D_{ m}^{\frac{mp}2} + O(R^{-Np})  \right)  \Pi_{j=1}^k \| f_{j} \|_{L^2}^p.
 \end{split}
 \]
Provided that the power of $D_{ m+1}$ is negative, and by taking $D_{ m +1}$ large enough relative to the constant involved and $D_{ m}$, the argument gives us a good estimate:
\begin{equation} \label{gest2}
 \mu(B_R)^\frac1p \leq ( \frac12 A(R) + O(R^{-N})) \Pi_{j=1}^k \| f_{j} \|_{L^2};
\end{equation}
In justifying the above we have also used that $\mu_{m}(N_{m}) \gtrsim \frac{\mu(B_R)}{D_m^{\kappa_m}}$.
For the power of $D_{ m+1}$ to be negative, it suffices to have
\[
p > \frac{2(n- m)}{(n- m-1)k + m}. 
\]
We are given that $p > \frac{2(n+k)}{k(n+k-2)}$, therefore it suffices to verify that
\[
\frac{2(n+k)}{k(n+k-2)}> \frac{2(n- m)}{(n- m-1)k + m}.
\]
A little calculus shows that this is true given that $n >k$.  
 
Thus, under the hypothesis that at least $m$ waves are transversal, we have achieved our goal (for this subsection) of either providing a good estimate \eqref{gest2}
or by reducing the dimensionality of the relevant algebraic variety.

\subsection{The algebraic case - second take} \label{NCC}  Here we look into the cases which have not been treated in the above subsection. The setup is described at the beginning of the previous subsection, just before \eqref{car1}. Recall that we are in the algebraic case and in this subsection we consider the remaining cases when on the right-hand side of $\eqref{car1}$ there are at most $m-1$ waves which are transversal, that is the set of $i$ with $car_i=trans$ has cardinality at most $m-1$. 

In fact, by choosing $\gamma_0$ small enough with respect to the parameter $\nu$ in \eqref{trans}, it follows that there are exactly $m-1$ waves that are transversal and $n-m$ that are non-transversal. We recall here that we work with $m \leq n-1$, thus we have at least one wave whose main contribution comes from non-transversal packets; otherwise the arguments below would be vacuous. To keep notation simple, we drop the subscripts notation $trans, ntrans$ from the functions involved and simply keep in mind that the first $m-1$ waves are transversal and the last $n-m$ are not. 

We decompose $B_R$ into balls $B_l$ of radius $\rho$ with 
\[
\rho^{\frac12+\delta_{1}}=R^{\frac12+\delta_0},
\]
where $\delta_1$ is a parameter chosen such that $\delta_0 \ll \delta_1\ll 1$. 
For each $l$ and $i \in \{1,..,k\}$ we select the tubes from $\calT_i$ that intersect $B_l \cap N_{R^{\frac12+\delta_0}}(Z) \cap N_m$, that is
\[
\calT_{i,l}= \{ T \in \calT_i | T \cap N_{R^{\frac12+\delta_0}}(Z) \cap N_m \cap B_l \ne \emptyset \},
\]
and let $f_{i,l}= f_{i, \calT_{i,l}}$ whose definition we recall here
\[
f_{i,\calT_{i,l}}= \sum_{T \in \calT_{i,l}} f_{i,T}.
\]
The following holds true:
\[
\| \Pi_{i=1}^k \calE_i f_i \|_{L^p(B_R \cap N_m)}^p  \les \sum_l \| \Pi_{i=1}^k \calE_i f_{i,l} \|_{L^p(B_l \cap N_m)}^p + O(R^{-N}) \Pi_{i=1}^k \| f_i \|_{L^2}^p, 
\]
since the tubes that do not intersect $N_m$ have their contribution estimated by $O(R^{-N})$.

The relevant tubes are divided into two classes: tangential and non-tangential; this is done following Section \ref{sat}. A tube $T$ is called tangential to $Z$ in $B_l$ if for any $x \in T$ and $y \in Z \cap 2B_l$ with 
$|x-y| \leq 2 \rho^{\frac12+\delta_1}$,
\[
\angle(T,T_y Z) \leq \rho^{-\frac12+\delta_1}.
\]
As we already mentioned in section \ref{sat} it follows that $T \cap B_l \subset N_{2 \rho^{\frac12+\delta_1}} Z \cap B_l$. 

A tube is called non-tangential to $Z$ in $B_l$ if it is not tangential, that is there exists $x \in T$ and $y \in Z \cap 2B_l$ with $|x-y| \leq 2 \rho^{\frac12+\delta_1}$ such that $\angle(T,T_y Z) > \rho^{-\frac12+\delta_1}$. 
This implies 
\[
N_{2 \rho^{\frac12+\delta_1}} (T) \cap 2B_l \cap Z_{v(T), > \rho^{-\frac12+\delta_1}} \neq \emptyset. 
\]
The set of tangent tubes to $Z$ in $B_l$ is denoted by $\calT_{i,l,tang}$.  The complement set in $\calT_{i,l}$ is denoted by $\calT_{i,l,ntang}$. 
Using these families of waves packets, we then define
\[
f_{i,l,tan}= f_{i,\calT_{i,l,tang}}, \qquad f_{i,l,ntang}= f_{i,\calT_{i,l,ntang}}.
\]

Let us look into how the new concepts of tangential/non-tangential tubes relates to the previously introduce concepts of transversal/non-transversal tubes. The first observation is that for $R$ large enough, a tube $T$ that is tangent to $Z$ in $B_l$ cannot be transversal and vice-versa. Thus the first $m-1$ waves contain only tubes that are non-tangential (or with negligible contribution in $B_l$).

The rest of the waves contain only tubes that are non-transversal, and they can be either tangential or non-tangential to $Z$ in $B_l$. A priori the largest angle that a non-tangential tube can make with $Z$ can take any value, but in the current context things are a bit more rigid. Precisely, within $N_m$, the angle of such a tube with $Z$ is bounded from above by $2 \gamma_0$.

Then we continue with
\[
\sum_l \| \Pi_{i=1}^k \calE_i f_{i,l} \|_{L^p(B_l \cap N_m)}^p \lesssim \sum_l  \sum_{car} \| \Pi_{i=1}^k \calE_i f_{i,l,car_i} \|_{L^p(B_l \cap N_m)}^p. 
\]
where $\sum_{car(i)}$ runs over all possible combinations $car=(car(1),..,car(k)) \in \{ntang \}^{m-1} \times \{tang,ntang\}^{n-m}$. 
One of these combinations on the right-hand side is dominant, that is there is a choice $car \in \{ort\}^{ m - 1} \times \{tan,obl\}^{n-m}$ such that
\[
\sum_l \| \Pi_{i=1}^k \calE_i f_{i,l} \|_{L^p(B_l \cap N_m)}^p \les  \sum_l \| \Pi_{i=1}^l \calE_i f_{i,l,car_i} \|_{L^p(B_l \cap N_m)}^p, 
\]
and it suffices to estimate this combination. We encounter two scenarios: $car(i)=ntang$ for all $i=1,..,k$ or $car(i)=tang$ for at least one $i \in \{1,..,k\}$. 
The following two subsection will address these two scenarios. 

\subsection{Non-tangential tubes dominate} \label{nttd} All tubes are transversal, that is $car \in \{ntang\}^k$. From Lemma \ref{Zan} it follows that a tube $T \in \calT_i$ can be non-tangential to $Z$ in at most $D_m^n$ balls $B_l$, therefore if $car(i) \in \{ntang\}$, the following holds true:
\begin{equation} \label{trineq}
\sum_l \| f_{i,l,car(i)} \|_{L^2}^2 \les D^n_m \| f_i \|_{L^2}.
\end{equation}

The induction hypothesis on each ball $B_l$ gives us:
\[
\| \Pi_{i=1}^k \calE_i f_{i,l, car(i) } \|_{L^p(B_l \cap N_m) } \leq A(\rho) \Pi_{i=1}^k \| f_{i,l,car(i)} \|_{L^2(U_i)} + O(R^{-N}) \Pi_{i=1}^k \| f_{i} \|_{L^2(U_i)}.  
\]
We sum with respect to $l$ to obtain:
\[
\begin{split}
\left( \sum_l \| \Pi_{i=1}^k \calE_i f_{i,l, car(i)} \|^p_{L^p(B_R \cap N_m)} \right)^\frac1p \leq &  A(\rho) \left(  \sum_l \Pi_{i=1}^k \| f_{i,l,car(i)} \|_{L^2(U_i)}^p \right)^\frac1p + O(R^{-N}) \Pi_{i=1}^k \| f_{i} \|_{L^2(U_i)}  \\
\lesssim &  A(\rho) D_m^{\frac{nk}2}  \Pi_{i=1}^k \| f_{i} \|_{L^2(U_i)} + O(R^{-N}) \Pi_{i=1}^k \| f_{i} \|_{L^2(U_i)},
\end{split}
\]
where, in passing to the last line, we have used \eqref{trineq} to obtain
\[
\begin{split}
\left( \sum_l \Pi_{i=1}^k \| f_{i,l,car(i)} \|_{L^2(U_i)}^{\frac2k} \right)^{\frac{k}2} \lesssim \Pi_{i=1}^k  \left( \sum_l \| f_{i,l,car(i)} \|_{L^2(U_i)}^2 \right)^{\frac12} \lesssim D_m^{\frac{nk}2} \Pi_{i=1}^k \| f_{i} \|_{L^2(U_i)},
\end{split}
\]
which then is interpolated ($\frac2k < p < \infty$) with the trivial inequality
\[
\max_l \Pi_{i=1}^k \| f_{i,l,car(i)} \|_{L^2(U_i)} \les \Pi_{i=1}^k \| f_{i} \|_{L^2(U_i)}. 
\]
Therefore in this case we conclude with the following inequality:
\begin{equation} \label{gest3}
\mu(B_R)^\frac1p \lesssim \left( D_m^{\kappa_m + \frac{nk}2} A(\rho) + O(R^{-N}) \right) \Pi_{i=1}^k \| f_{i} \|_{L^2(U_i)}. 
\end{equation}

\subsection{At least one family of tubes is tangent} Here we estimate the following term
\[
\sum_l  \sum_{car} \| \Pi_{i=1}^k \calE_i f_{i,l,car_i} \|_{L^p(B_R \cap N_m)}^p
\]
where $car(m)= tan$; this choice does not restrict the generality of the argument. We recall that the first $m-1$ waves are transversal, hence non-tangent, the $m$'th wave is tangent, while the other waves are non-transversal, but each can be either tangential or non-tangential. 

In the previous two sections we have paid a lot of attention to the new smaller scale $\rho$; in this section the smaller scale play little role - the angular definition of tangent tubes contains the $\rho$ scales. Other than that the focus will be on balls of radius $R^{\frac12+\delta_0}$. 
Consider a ball $B_{R^{\frac12+\delta_0}}(x_0) \subset N_m$. 
With $V$ being the $n-m$ plane that has been used in the definition of $N_m$, we recall the following geometric information that we have so far for the current setup:

i) for any $i \in \{1,..,m-1\}$ and for any $T \in \calT_i$, $\angle(v(T),V) > 4 \gamma_0$;

ii) for any $i \in \{m+1,..,n-1\}$ and for any $T \in \calT_i$, $\angle(v(T),V) \leq 4 \gamma_0$;

iii) for any $T \in \calT_{m}$ with $T \cap B_{R^{\frac12+\delta_0}}(x_0) \neq \emptyset$, $\angle(v(T), T_{z_0} Z) \leq \rho^{-\frac12+\delta_1}$, for some $z_0 \in B_{R^{\frac12+\delta_0}}(x_0) \cap Z$;

iv) $\angle(V,T_{z_0} Z) \leq \gamma_0$.

We let
\[
f_{m,l,tan,B_{R^{\frac12+\delta_0}(x_0)}} = \sum_{T \in \calT_{m,l,tan}: T \cap B_{R^{\frac12+\delta_0}(x_0)} \neq \emptyset} f_{m,T}. 
\]
The above geometric statements imply that $\mathcal{F} (\calE_m f_{m,l,tan,B_{R^{\frac12+\delta_0}(x_0)}})$ is supported in the following subset of $S_m$
\[
\tilde S_m = \{ \zeta \in S_m: | \pi_{(T_{z_0} Z)^\perp} N_m(\zeta)| \leq \rho^{-\frac12+\delta_1}  \}.
\]
We intend to invoke Lemma \ref{kL}, but some more work is needed in that direction. To put things in perspective, $(T_{z_0} Z)^\perp$
has dimension $m$, thus hinting at localization properties in $m$ directions; but some of the directions in $(T_{z_0} Z)^\perp$ can be very close
to a linear combination of $N_1(\zeta_1),..,N_{m-1}(\zeta_{m-1})$  making them useless for the purpose of extracting good localization information. However it is clear that there is at least one good direction based on dimensionality considerations. 

We make a fixed choice of $\bar \zeta_i \in S_i, i \in \{1,..,k\} \setminus \{m\}$, but let $\zeta_m$ vary inside $S_m$. \eqref{curva} gives us:
\begin{equation}
| N_1(\bar \zeta_1) \wedge ... \wedge N_{k}(\bar \zeta_{k}) \wedge S_{N_m(\zeta_m)} v | \geq \nu_1 | v|,
\end{equation}
for every $v \in span(N_1(\bar \zeta_1), ..,  N_{k}(\bar \zeta_{k}))^\perp$. 
Therefore if we let $V_{\zeta_m}= span(\pi_{T_{\zeta_m} S_m} N_1(\bar \zeta_1), .., \pi_{T_{\zeta_m} S_m} N_{k}(\bar \zeta_k))$, we see that \eqref{Sv} is satisfied. 
This implies that $V_{\zeta_m}^\perp=span(\pi_{T_{\zeta_m} S_m} N_1(\bar \zeta_1), .., \pi_{T_{\zeta_{m}} S_{m}} N_{k}(\bar \zeta_{k}))^\perp$, 
where the orthogonal space was implicitly taken inside $T_{\zeta_m} S_m$. 

Next we show that \eqref{ph} is satisfied and for this we define $\calH^\perp=(T_{z_0} Z)^\perp \cap N_1(\bar \zeta_1)^\perp \cap .. \cap N_{m-1}(\bar \zeta_{m-1})^\perp$. $\calH^\perp$ is easily seen to be a one-dimensional subspace: indeed it is orthogonal to $N_1(\bar \zeta_1),..,N_{m-1}(\bar \zeta_{m-1})$ and "almost" orthogonal to $N_m(\zeta_m),..,N_{k}(\bar \zeta_{k})$ since it is orthogonal to  $T_{z_0} Z$.   We are going to weaken our localization property in that 
$\mathcal{F} (\calE_m f_{m,l,tan,B_{R^{\frac12+\delta_0}}(x_0)})$ is supported in the following subset of $S_m$
\[
\tilde S_m=  \{ \zeta \in S_m: | \pi_{\calH^\perp} N_m(\zeta)| \leq \rho^{-\frac12+\delta_1}  \}.
\]
Let $e \in \calH^\perp$ be the unit vector. Since $\calH^\perp$ is one-dimensional, it suffices to establish \eqref{ph} for $e$.  First we have that $|\angle(e,N_m(\zeta_m)) - \frac{\pi}2| \lesssim \rho^{-\frac12+\delta_1} \ll \nu$. Second, for each $i \in \{1,..,m-1\}$ we know that $e \perp N_i(\bar \zeta_i)$; from this it follows that  $|\angle (e,N_i(\zeta_i))-\frac{\pi}2| \leq c \ll 1$ for any choice $\zeta_i \in S_i,  i=1,..m-1$ (this is a consequence of the small diameter of each $S_i$, see \eqref{cN}). Third, for each $i \in \{m+1,..,k\}$, it follows that $|\angle (e,N_i(\bar \zeta_i))-\frac{\pi}2| \les \gamma_0$. Thus, by choosing $c$ and $\gamma_0$ small enough, \eqref{ph} is guaranteed to hold true.

At this point we can invoke Lemma \ref{kL} to conclude that there exists $\tilde S_m'$ a submanifold of $S_m$ of co-dimension one such that 
\[
\tilde S_m \subset B_{c_1\rho^{-\frac12+\delta_1}}(\tilde S_m'). 
\]
for some $c_1 \les 1$. In addition, $|N_{\zeta_m} S_m' \wedge N_1(\zeta_1) \wedge ... \wedge N_{k}(\zeta_{k})| \ges 1$ (where we skip $N_m(\zeta_m)$)
for any choice $\zeta_i \in S_i, i \in \{1,..,k\} \setminus \{ m\}$ and $\zeta_m \in \tilde S_m'$; this is first obtained for $\bar \zeta_1, .., \bar \zeta_k$ and the extended
to all choices $\zeta_i$ by using the smallness of the diameter of each $S_i$. 

At this time we can invoke Theorem \ref{MB} to obtain the estimate
\[
\| \Pi_{i=1}^{k} \calE_i f_{i,l,car(i)} \|_{L^\frac{2}{k-1}(B_{R^{\frac12+\delta_0}}(x_0))} \lesssim   (R^{\frac12+\delta_0})^\epsilon (\rho^{-\frac12+\delta_1})^\frac12 
\Pi_{i=1}^k \| f_{i,l,car(i)}  \|_{L^2(U_i)},
\]
where we use the convention that $f_{m,l,tan}$ is replaced by $f_{m,l,tan,B_{R^{\frac12+\delta_0}(x_0)}}$. In fact we can do better in several place.
We can drop the index $l$ and $car(i)$ on the right-hand side and replace each $f_{i}$
by $f_{i,B_{R^{\frac12+\delta_0}}(x_0)}$ where
\[
f_{i,B_{R^{\frac12+\delta_0}}(x_0)}= \sum_{T \in \calT_i: T \cap B_{R^{\frac12+\delta_0}}(x_0) \neq \emptyset} f_{i,T},
\]
thus we have
\[
\begin{split}
\| \Pi_{i=1}^{k} \calE_i f_{i,l,car(i)} \|_{L^\frac{2}{k-1}(B_{R^{\frac12+\delta_0}}(x_0))} & \lesssim_{\epsilon}  (R^{\frac12+\delta_0})^\epsilon
(\rho^{-\frac12+\delta_1})^\frac12 \Pi_{i=1}^k \| f_{i,B_{R^{\frac12+\delta_0}}(x_0))} \|_{L^2(U_i)} \\
& + O(R^{-N}) \Pi_{i=1}^k \| f_i\|_{L^2(U_i)}.
\end{split}
\]
We have implicitly used the fact the dual scale of the localization is much smaller than the scale of the ball where the estimate is performed; 
indeed the localization comes at scale $\rho^{-\frac12+\delta_1}$ whose dual scale is 
$\rho^{\frac12-\delta_1}=\rho^{\frac12+\delta_1-2\delta_1}=R^{\frac12 +\delta_0} \rho^{-2\delta_1} \ll R^{\frac12 +\delta_0}$, and $R^{\frac12 +\delta_0}$
 is the scale of the ball where we perform the estimate. 

To conclude the main ingredient in this argument, the left-hand side above needs to be summed up in $l^{\frac2{k-1}}$ with respect to a relevant set of balls $B_{R^{\frac12+\delta_0}}(x_0)$ that cover the set $N_m$:
\[
\| \Pi_{i=1}^{k} \calE_i f_{i,l,car(i)} \|_{L^\frac{2}{k-1}(N_m)} \lesssim  (R^{\frac12+\delta_0})^\epsilon (\rho^{-\frac12+\delta_1})^\frac12 
\left( \sum_{x_0} \Pi_{i=1}^k \| f_{i,B_{R^{\frac12+\delta_0}}(x_0))} \|_{L^2(U_i)}^{\frac2{k-1}} \right)^{\frac{k-1}2}.
\]
where $x_0$ is chosen such that the set $\{ B_{R^{\frac12+\delta_0}}(x_0)) \}_{x_0}$ has the finite intersection property and it covers $N_m$.

At this point we claim the following result 
\begin{equation} \label{aux30}
\left( \sum_{x_0 \in R^{\frac12+\delta_0} \Z^n: |x_0| \lesssim R} \Pi_{i=1}^k \| f_{i,B_{R^{\frac12+\delta_0}}(x_0)} \|_{L^2(U_i)}^{\frac2{k-1}} \right)^{\frac{k-1}2} \lesssim_{\epsilon}
(\frac{R}{R^{\frac12+\delta_0}})^\epsilon \Pi_{i=1}^k \| f_{i} \|_{L^2(U_i)}.
\end{equation}
This result essentially tells us that we can sum the right-hand side of the previous inequality over a relevant set of balls $B_{R^{\frac12+\delta_0}}(x_0)$ that cover that cover $B_R$. Taking \eqref{aux30} for granted (we will address this at the end of this section), allows us to conclude with
\[
\| \Pi_{i=1}^{n-1} \calE_i f_{i,l,car(i)} \|_{L^\frac{2}{n-2}(N_m)} \lesssim \left( R^\epsilon (\rho^{-\frac12+\delta_1})^\frac12 + O(R^{-N}) \right) \Pi_{i=1}^{n-1} \| f_{i} \|_{L^2(U_i)},
\]
where we recall that $k=n-1$. 

At the same time we have the trivial estimate:
\[
\| \Pi_{i=1}^{n-1} \calE_i f_i \|_{L^\frac2{n-1}(B_R)} \les  R^{\frac{n-1}2} \Pi_{i=1}^{n-1} \| f_i \|_{L^2(U_i)}.
\]
Then we interpolate the estimates in $L^{\frac2{n-2}}$ with $L^\frac2{n-1}$ with $\theta=\frac{2(n-1)}{2n-1}$ and $1-\theta=\frac{1}{2n-1}$ respectively,
to obtain an estimate in  $L^{\frac{2(2n-1)}{(n-1)(2n-3)}}$ with the factor
\[
(R^\epsilon)^{\frac{2(n-1)}{2n-1}} (\rho^{-\frac12+\delta_1})^{\frac12 \cdot \frac{2(n-1)}{2n-1}}(R^{\frac{n-1}2})^{\frac1{2n-1}} \lesssim 
R^{\epsilon + O(\delta_0 +\delta_1^2)}.
\]
This last inequality suggests that a good choice of parameters is
\[
\delta \ll \delta_0 \ll \delta_1 \ll \epsilon. 
\]
This leads to the following inequality:
\begin{equation} \label{gest4}
\mu(B_R)^\frac1p \leq C(\epsilon) D_m^{\kappa_m} R^{\frac32\epsilon}  \Pi_{i=1}^k \| f_{i} \|_{L^2(U_i)},
\end{equation}
where $\epsilon$ and $C(\epsilon)$ are the parameters given by Theorem \ref{MB} and \eqref{aux30} - as about $C(\epsilon)$ we pick the maximum of the two used in those results.

\subsection{The special case $m=n$.} \label{special} In the above analysis we have not considered the special case $m=n$, see the discussion at the begining of Section \ref{pp2}. Precisely we look at the following situation. We are given an algebraic variety $Z=Z(P_1,..,P_{n})$ which is a transverse complete intersection. The polynomials $P_i$'s have degrees less than $D_i$, with $D_1 \leq D_2 \leq .. \leq D_n$ and the following holds true:
\begin{equation} \label{alg4a}
\mu(N_{R^{\frac12+ \delta_0}}(Z)  \cap B_R) \gtrsim \frac{\mu(B_R)}{D_n^{\kappa_n}}.
\end{equation} 

From Theorem \ref{ftci} it follows that $N_{R^{\frac12+ \delta_0}}(Z)$ is a finite union of a collection of balls $\{ B_{R^{\frac12+\delta_0}}(x_i) \}_{i=1}^N$ 
where $N \leq \Pi_{i=1}^n D_i \leq D_n^n$. In each of this balls we have
\[
\mu(B_{R^{\frac12+\delta_0}}(x_i))^\frac1p \leq A(R^{\frac12+\delta_0}) \Pi_{i=1}^k \| f_i \|_{L^2(U_i)}. 
\]
Thus
\[
\mu(N_{R^{\frac12+ \delta_0}}(Z))^\frac1p \lesssim N^\frac1p A(R^{\frac12+\delta_0}) \Pi_{i=1}^k \| f_i \|_{L^2(U_i)}, 
\]
from which we can conclude with
\begin{equation} \label{gest5}
\mu(B_R)^\frac1p \lesssim D_n^{\frac{\kappa_n+n}p} A(R^{\frac12+\delta_0}) \Pi_{i=1}^k \| f_i \|_{L^2(U_i)}. 
\end{equation}

\subsection{Conclusion} \label{concl}
In this section we bring together all the bounds we collected and conclude with the desired bound. 
It is obvious that for $R \les 1$ we have $A(R) \les 1$.  

Putting together all the estimates \eqref{gest1}, \eqref{gest2}, \eqref{gest3}, \eqref{gest4} and \eqref{gest5}, we conclude with
\[
A(R) \leq \max( \frac12 A(R) + O(R^{-N}), C A(R^{\frac{\frac12+\delta_0}{\frac12+\delta_1}}) + O(R^{-N}), C A(R^{\frac12+\delta_0}), C \cdot C(\epsilon)  R^{\frac32\epsilon}). 
\]
Here $C=C(D_1,..,D_n,\kappa_1,..,\kappa_n)$ with the observation that the degrees $D_i$ and powers $\kappa_i$ are universal, independent on $R$. 
From this inequality we can conclude with
\[
A(R) \leq \tilde C(\epsilon) R^{2\epsilon},
\]
for $R \gg 1$. This finishes the proof of Proposition \ref{mainT}, modulo the claim \eqref{aux30} which is discussed in the next section.

We also need to go back at the very beginning of this argument, see the commentaries following Proposition \ref{indp}, and address the issue of how the proof needs to be modified when taking into account the technical layer that the wave margin brings in. The current argument reveals two main ways in which we invoke the induction
hypothesis: 

i) we use information at scale $R$ to bootstrap information at the same scale $R$ - this is used in the non-algebraic case;

ii) we use information at scale $R^{\frac{1+2\delta_0}{1+2\delta_1}}$ to obtain information at scale $R$, where $R^{\frac{1+2\delta_0}{1+2\delta_1}} \ll R$ - this is used in the algebraic case when all the tubes are non-tangential, see Section \ref{nttd}. 

The use of wave packets perturbs the margin of each wave by a factor of $\approx R^{-\frac{1+\delta}2}$. Then it is easily seen that for ii) we have that the new margin 
is bounded by
\[
M- R^{-\frac14} - C  R^{-\frac{1+\delta}2} < M - (R^{\frac{1+2\delta_0}{1+2\delta_1}})^{-\frac14} .
\]
This suffices for a clean integration of the margin concept into the argument in the case ii). 

A bit of care is needed for i), since we use a bootstrap type argument and the margin is not amenable to such an argument, given that it will be modified. The easy way 
to fix this is as follows: instead of using a bootstrap type argument, we simply use a direct induction on scale argument which uses the information at scale $\frac{R}2$ to conclude with information at scale $R$ and the simple observation that 
\[
A(R) \leq C A(\frac{R}2),
\]
with an explicit constant $C$ (which is independent of $R$) that can be later absorbed in the argument. This modifies the above inequality to 
\[
A(R) \leq \max( \frac12 A(\frac{R}2) + O(R^{-N}), C A(R^{\frac{\frac12+\delta_0}{\frac12+\delta_1}}) + O(R^{-N}), C A(R^{\frac12+\delta_0}), C \cdot C(\epsilon)  R^{\frac32\epsilon}),
\]
from which we can conclude the argument as above. Concerning the margin we note that the simple inequality
\[
M- R^{-\frac14} - C  R^{-\frac{1+\delta}2} < M - (\frac{R}2)^{-\frac14},
\]
allows for a clean integration of the margin concept in the argument. 

\subsection{Argument for \eqref{aux30}}

In this section we address the claim we made in \eqref{aux30}:
\[
\left( \sum_{x_0 \in R^{\frac12+\delta_0} \Z^n: |x_0| \lesssim R} \Pi_{i=1}^k \| f_{i,B_{R^{\frac12+\delta_0}}(x_0))} \|_{L^2(U_i)}^{\frac2{k-1}} \right)^{\frac{k-1}2} \lesssim_{\epsilon}
(\frac{R}{R^{\frac12+\delta_0}})^\epsilon \Pi_{i=1}^k \| f_{i} \|_{L^2(U_i)}.
\]
It is important to note that this is at the level of a generic estimate, that is we use only the transversality condition \eqref{trans} and there is no need to use any curvature information coming from \eqref{curva}. In addition, the fact that $k=n-1$ plays no role in this estimate, and we can work with any $1 \leq k \leq n$. 

Without restricting the generality of the argument, we can assume that each $S_i$ has a parametrization of type 
$\xi_i=\varphi_i(\xi_1,..,\hat \xi_i,..,\xi_n )$. This allows us to use the conservation law
\[
\| \calE_i f_i \|_{L^2(x_i=t)} = \| \calE_i f_i \|_{L^2(x_i=0)} \approx \| f_i \|_{L^2(U_i)}, \quad \forall t \in \R. 
\]
The above $L^2(x_i=0)$ should be understood as follows:
\[
\| F \|_{L^2(x_i=t)}^2 = \int |F(x_1,..,x_{i-1},t,x_{i+1},..,x_n)|^2 dx_1 ... d \hat x_i .. dx_n. 
\]
Then \eqref{aux30} follows from the following inequality:
\begin{equation} \label{aux31}
\left( \sum_{x_0 \in R^{\frac12+\delta_0} \Z^n: |x_0| \lesssim R} \Pi_{i=1}^k \| \chi_{B_{R^{\frac12+\delta_0}}(x_0))}  \calE_i f_i\|_{L^2(x_i=x_{0,i})}^{\frac2{k-1}} \right)^{\frac{k-1}2} \lesssim_{\epsilon}
(\frac{R}{R^{\frac12+\delta_0}})^\epsilon \Pi_{i=1}^k \| f_{i} \|_{L^2(U_i)}.
\end{equation}
Now \eqref{aux31} reads as follows: we collect the mass of each wave $\calE_i f_i$ in $B_{R^{\frac12+\delta_0}}(x_0))$ (measured on appropriate hyperplanes passing through the center of the ball), multiply these masses and control their $l^{\frac2{k-1}}$ norm with respect to the balls. 
This estimate is "morally" at the same level with the $L^{\frac2{k-1}}$ estimate for $\calE_i f_i$. Indeed, if one follows the proof provided by the 
author in \cite{Be1} for the multilinear restriction estimate in $L^{\frac2{k-1}}$ the argument uses an induction on scales technique:
we use the induction on a smaller ball, identify the component of the mass of each wave that is relevant to that mass (which is very similar to 
$\| \chi_{B_r (x_0)}  \calE_i f_i \|_{L^2(x_i=x_{0,i})}$)  and then perform an $l^{\frac2{k-1}}$ estimate on the product of these masses - to obtain 
an estimate similar to \eqref{aux31}. We leave the details to the interested reader. 

A more complicated version of \eqref{aux31} is provided in Theorem 7.1 in \cite{Be4}, and the analogy with \eqref{aux31} is highlighted in the remarks following the statement of the theorem there. However, the setup in \cite{Be4} is far more complex and, although doable, it may be difficult to deduct (the easier) \eqref{aux31} from the (more complicated) arguments used in the proof of Theorem 7.1 in  \cite{Be4}.

\bibliographystyle{amsplain} \bibliography{HA-refs}

\end{document}